\documentclass{amsart}
\usepackage{amssymb}
\usepackage{url}
\usepackage{enumerate}
 \usepackage{hyperref}
\usepackage{fancyhdr}

\newtheorem{theorem}{Theorem}[section]
\newtheorem{lemma}[theorem]{Lemma}
\newtheorem{conjecture}[theorem]{Conjecture}
\newtheorem{proposition}[theorem]{Proposition}

\theoremstyle{definition}

\theoremstyle{remark}
\newtheorem{remark}[theorem]{Remark}
\newtheorem{remarks}[theorem]{Remarks}
\numberwithin{equation}{section}
\def\ctl#1{

\smallskip\smallskip
\centerline{#1}

\smallskip\smallskip\noindent\!\!}
\def\ctlm#1{

\smallskip
\centerline{#1}

\smallskip\noindent\!\!}
\let\ds=\displaystyle
\let\wt=\widetilde
\let\ov=\overline
\def\N{\mathbb{N}}
\def\M{\mathbb{M}}
\def\R{\mathbb{R}}
\def\C{\mathbb{C}}
\def\Q{\mathbb{Q}}
\def\Z{\mathbb{Z}}

\def\Cl{{\mathcal C}\hskip-2pt{\ell}}
\def\Frac#1#2{\hbox{\footnotesize $\displaystyle \frac{#1}{#2}$}}
\def\plus{\displaystyle\mathop{\raise 2.0pt \hbox{$\bigoplus $}}\limits}
\def\prd{ \displaystyle\mathop{\raise 2.0pt \hbox{$\prod$}}\limits}
\def\sm{  \displaystyle\mathop{\raise 2.0pt \hbox{$\sum$}}\limits}
\def\lien{\mathrel{\mkern-4mu}}
\def\too{\relbar\lien\rightarrow}
\def\tooo{\relbar\lien\relbar\lien\too}

\def\order{\raise1.5pt \hbox{${\scriptscriptstyle \#}$}}

\begin{document}

\title[Heuristics in direction of a $p$-adic Brauer--Siegel theorem]
{Heuristics and conjectures in direction of  \\ a $p$-adic Brauer--Siegel theorem}

\author{Georges Gras}
\address{Villa la Gardette, Chemin Ch\^ateau Gagni\`ere 
 F--38520 Le Bourg d'Oisans, France,
{\rm \url{https://www.researchgate.net/profile/Georges_Gras}}}
\email{g.mn.gras@wanadoo.fr}


\date{February 4, 2018 -- Final version after corrections: March 1, 2018}

\begin{abstract}
 Let $p$ be a fixed prime number. 
Let $K$ be a totally real number field of discriminant $D_K$ and
let ${\mathcal T}_K$ be the torsion group of the Galois group of  the maximal 
abelian $p$-ramified pro-$p$-extension of $K$ (under Leopoldt's conjecture). 
We conjecture the existence of a constant ${\mathcal C}_p>0$ such that 
${\rm log}(\order {\mathcal T}_K) \leq {\mathcal C}_p \cdot {\rm log}(\sqrt{D_K})$ 
when $K$ varies in some specified families (e.g., fields of fixed degree). 
In some sense, we suggest the existence of a $p$-adic analogue, 
of the classical Brauer--Siegel Theorem, wearing here on the valuation of 
the residue at $s=1$ (essentially equal to $\order {\mathcal T}_K$)
of the $p$-adic $\zeta$-function $\zeta_p(s)$ of $K$.
We shall use a different definition that of Washington, given in the 1980's,
and approach this question via the arithmetical study of ${\mathcal T}_K$
since $p$-adic analysis seems to fail because of possible abundant ``Siegel zeros''
of $\zeta_p(s)$, contrary to the classical framework.
We give extensive numerical verifications for quadratic and cubic fields 
(cyclic or not) and publish the PARI/GP programs directly usable by the reader
for numerical improvements. 
Such a conjecture (if exact) reinforces our conjecture that any fixed number 
field $K$ is $p$-rational (i.e., ${\mathcal T}_K=1$) for all $p \gg 0$.
\end{abstract}

\maketitle

\section{Abelian $p$-ramification -- Main definitions and notations}\label{def}

Let $K$ be a totally real number field of degree $d$, and let $p\geq 2$ be 
a prime number fulfilling the Leopoldt conjecture in $K$.
We denote by $\Cl_K$ the $p$-class group of $K$ (ordinary sense)
and by $E_K$ the group of $p$-principal global units $\varepsilon \equiv 1 
\pmod {\prod_{{\mathfrak p} \mid p} {\mathfrak p}}$ of $K$.

\smallskip
Let's recall from \cite{Gr1,Gr4} the diagram of the so called 
{\it abelian $p$-ramification theory}, in which 
$K^{\rm c} = K \Q^{\rm c}$ is the cyclotomic $\Z_p$-extension of $K$
(as compositum with that of~$\Q$), 
$H_K$ the $p$-Hilbert class field and $H_K^{\rm pr}$ the 
maximal abelian $p$-ramified (i.e., unramified outside $p$)
pro-$p$-extension of~$K$.

\smallskip
Let $U_K:=\plus_{{\mathfrak p} \mid p} U_{\mathfrak p}^1$ 
be the $\Z_p$-module (of $\Z_p$-rank $d$) of $p$-principal 
local units of $K$, where each $U_{\mathfrak p}^1
:=\{u \in K_{\mathfrak p}^\times, \ u \equiv 1 \pmod {\ov {\mathfrak p}} \}$ 
is the group of $\ov{\mathfrak p}$-principal units of the 
completion $K_{\mathfrak p}$ of $K$ at ${\mathfrak p} \mid p$,
where $\ov {\mathfrak p}$ is the maximal ideal of the ring 
of integers of $K_{\mathfrak p}$.

\smallskip
For any field $k$, let $\mu^{}_k$ be the group of roots of unity of $k$
of $p$-power order.
Then put $W_K := {\rm tor}^{}_{\Z_p} \big (U_K  \big )  = 
\plus_{{\mathfrak p} \mid p} \mu^{}_{K_{\mathfrak p}}$ and 
${\mathcal W}_K := W_K /\mu^{}_K$, where $\mu^{}_K=\{1\}$ or $\{\pm 1\}$. 

Let $\overline E_K$ be the closure in $U_K$ of the diagonal image 
of $E_K$; by class field theory this gives in the diagram 
${\rm Gal}(H_K^{\rm pr}/H_K) \simeq U_K/\overline E_K$; then
let $\Cl_K^{\,\rm c}$ be the subgroup of $\Cl_K$ corresponding 
to the subgroup ${\rm Gal}(H_K/ K^{\rm c} \cap H_K)$. 

\smallskip
Put (see \cite[Chapter III, \S\,2\,(a) \& Theorem 2.5]{Gr1} with the set $S$ of infinite
places, to get the ordinary sense, and with the set $T$ of $p$-places):
\ctl{${\mathcal T}_K :=  {\rm  tor}^{}_{\Z_p}({\rm Gal}(H_K^{\rm pr}/K))
= {\rm Gal}(H_K^{\rm pr}/K^{\rm c})$.}
As we know, $\order {\mathcal T}_K$ is essentially the residue of the 
$p$-adic $\zeta$-function of $K$ at $s=1$ \cite{Col, Se}; we will detail 
this in Subsection \ref{places}.

\smallskip
We have (because of Leopoldt's conjeture) the following exact sequence 
defining ${\mathcal R}_K$, where ${\rm log}_p$ is the $p$-adic logarithm
(\cite[Lemma III.4.2.4 \& Corollary III.3.6.3]{Gr1},
\cite[Lemma 3.1 \& \S\,5]{Gr4}):
\begin{equation*}
1\to  {\mathcal W}_K  \too 
 {\rm tor}_{\Z_p}^{} \big(U_K \big / \ov E_K \big) \\
 \mathop {\tooo}^{\!\!{\rm log}_p}  {\rm tor}_{\Z_p}^{}\big({\rm log}_p\big 
(U_K \big) \big / {\rm log}_p (\ov E_K) \big)=: {\mathcal R}_K \to 0. 
\end{equation*}

\noindent
The group ${\mathcal R}_K$ (or its order) is called the {\it normalized $p$-adic 
regulator of $K$} and makes sense for any number field (provided one replaces 
$K^{\,\rm c}$ by the compositum $\wt K$ of the $\Z_p$-extensions):
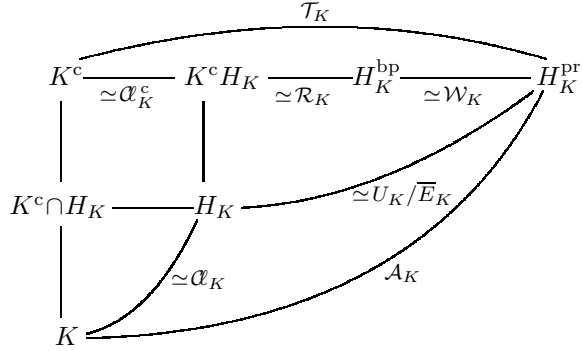
\unitlength=0.86cm 
$$\vbox{\hbox{\hspace{-2.8cm} 
 \begin{picture}(11.5,5.6)
\put(6.7,4.50){\line(1,0){1.3}}
\put(8.75,4.50){\line(1,0){2.0}}
\put(3.85,4.50){\line(1,0){1.4}}
\put(9.1,4.15){\footnotesize$\simeq\! {\mathcal W}_K$}
\put(4.3,2.50){\line(1,0){1.25}}

\bezier{350}(3.8,4.8)(7.6,5.8)(11.0,4.8)
\put(7.2,5.45){\footnotesize${\mathcal T}_K$}

\put(3.50,2.9){\line(0,1){1.25}}
\put(3.50,0.8){\line(0,1){1.4}}
\put(5.7,2.9){\line(0,1){1.25}}

\bezier{300}(3.9,0.55)(4.9,0.8)(5.6,2.3)
\put(5.2,1.3){\footnotesize$\simeq \! \Cl_K$}
\put(4.1,4.15){\footnotesize$\simeq\! \Cl_K^{\,\rm c}$}

\bezier{300}(6.3,2.5)(8.5,2.6)(10.8,4.3)
\put(8.0,2.6){\footnotesize$\simeq \! U_K/\ov E_K$}

\bezier{500}(3.9,0.48)(9,0.6)(10.95,4.3)
\put(8.5,1.4){\footnotesize${\mathcal A}_K$}

\put(10.85,4.4){$H_K^{\rm pr}$}
\put(5.4,4.4){$K^{\rm c} H_K$}
\put(8.0,4.4){$H_K^{\rm bp}$}
\put(6.8,4.14){\footnotesize$\simeq\! {\mathcal R}_K$}
\put(3.3,4.4){$K^{\rm c}$}
\put(5.55,2.4){$H_K$}
\put(2.7,2.4){$K^{\rm c} \!\cap \! H_K$}
\put(3.4,0.38){$K$}
\end{picture}   }} $$
\unitlength=1.0cm

\noindent
The field $H_K^{\rm bp}$, fixed by ${\mathcal W}_K$, is the 
Bertrandias--Payan field, i.e., the compositum of the $p$-cyclic extensions 
of $K$ embeddable in $p$-cyclic extensions of arbitrary large degree.

\section{$v$-adic analytic prospects}\label{sec4}

Let ${\mathcal K}_{\rm real}$ (resp. ${\mathcal K}_{\rm real}^{(d)}$) be the set 
of totally real number fields $K$ of any degree (resp. of fixed degree $d$).
For a fixed prime $p$ and a random $K \in {\mathcal K}_{\rm real}$, we have:
\ctl{$\order {\mathcal T}_K = \order \Cl_K^{\,\rm c} \cdot
\order {\mathcal R}_K \cdot \order {\mathcal W}_K$,} 
which may be equal to $1$ (defining ``$p$-rational fields'') or not, and it 
will be interesting to know if the $p$-adic valuation of $\order {\mathcal T}_K$ 
can be bounded according, for instance, to the discriminant $D_K$ of $K$. 
If so, this would be interpreted as a $p$-adic version of the 
archimedean Brauer--Siegel theorem, which is currently 
pure speculation, but we intend to experiment, algebraically, 
this context since $p$-adic analysis does not seem to succeed 
as explain by Washington in \cite{Wa2}: 

\medskip
\centerline{\it A Brauer--Siegel theorem using $p$-adic $L$-functions fails;} 

\medskip\noindent
in the same way, we have similar comments by Ivanov in \cite[Section 1]{Iv}:

\medskip
\centerline{\it The $p$-adic analogue of Brauer--Siegel and 
hence also of Tsfasman--Vlad\u{u}\c{t} fails.}

\medskip
But this requires some explanation:

\subsection{The Siegel zeros} 
In fact, there is a possible ambiguity about the definitions and the role of the 
discriminant in a $p$-adic Brauer--Siegel frame.

\smallskip
Let $K \in {\mathcal K}_{\rm real}$, let $h_K$ be its class number, 
$R_{K,p}$ its classical $p$-adic regulator, $D_K$ its discriminant; in \cite[\S\,3]{Wa2}, 
Washington considers a sequence of such number fields $K$, fulfilling the condition
 $\Frac{[K : \Q]}{v_p(\sqrt{D_K})} \to 0$, and study the limit:
\ctl{$\ds \lim_{K} \Big( \Frac{v_p(h_K \cdot R_{K,p})}{v_p(\sqrt{D_K})} \Big)$,} 
where $v_p$ denotes the $p$-adic valuation;
thus the above condition implies that $p$ must be ``highly ramified'' in the fields
of the sequence, which eliminates for instance families of fields of constant degree $d$. 
So, with Washington's definition, $K$ belongs in general to some towers of number 
fields (e.g., the cyclotomic one).

\smallskip
Washington shows examples and counterexamples of the $p$-adic 
Brauer--Siegel property $\Frac{v_p(h_K \cdot R_{K,p})}{v_p(\sqrt{D_K})} \to 1$
(\cite[Proposition 2 \& Theorem 2]{Wa2}). In his Theorem 3, he
uses the formula of Coates \cite[p. 364]{Coa}, which implies 
$\ds\liminf_K \Big(\Frac{v_p(h_K \cdot R_{K,p})}{v_p(\sqrt{D_K})} \Big) \geq 1$
as $\Frac{[K : \Q]}{v_p(\sqrt{D_K})} \to 0$.
We shall consider instead 
$\Frac{v_p \Big(h_K \cdot \frac{R_{K,p}}{\sqrt{D_K}}\Big) \cdot {\rm log}_\infty(p)}
{{\rm log}_\infty(\sqrt{D_K})}$, where ${\rm log}_\infty$ is the usual complex 
logarithm, or more precisely we shall study:
\ctl{$C_p(K):= \Frac{v_p(\order {\mathcal T}_K) \cdot 
{\rm log}_\infty(p)}{{\rm log}_\infty(\sqrt{D_K})} =
\Frac{{\rm log}_\infty(\order {\mathcal T}_K)}
{{\rm log}_\infty(\sqrt{D_K})}$, \ for any $K \in {\mathcal K}_{\rm real}$,}
then the existence of $\ds\sup_{K \in {\mathcal K}}(C_p(K))$, and of
$\ds\limsup_{K \in {\mathcal K}}(C_p(K))$, for any given infinite 
set ${\mathcal K} \subseteq {\mathcal K}_{\rm real}$, and $\ds\sup_p(C_p(K))$, 
$\ds\limsup_p(C_p(K)) \in \{0, \infty\}$ for $K$ fixed (see Conjectures 
\ref{conjprinc}, \ref{conjprinc2}).
However, there are some connections between the two definitions since
the quantity $v_p(h_K \cdot R_{K,p})$ appears in each of them; only the 
measure of the order of magnitude differs for the analysis of sequences of fields.
It is therefore not surprising to find, for instance in \cite{SSW, Wa2, Wa3},
some allusions to the group ${\mathcal T}_K$.

\smallskip
Let's finish these comments with a quote from Washington's paper illustrating the 
crucial fact that a great $v_p(\order {\mathcal T}_K)$ is related to the existence
of zeros, of the $p$-adic $\zeta$-function, 
or of the $L_p$-functions (see \cite{SSW, Wa3, Wa4, Wa5} for complements 
about these zeros and for some numerical data):

\medskip\noindent
{\it In the proof of the classical Brauer--Siegel theorem, one needs the
fact that there is at most one Siegel zero, that is, a zero close to $1$. 
The fact that the Brauer--Siegel theorem fails $p$-adically could be
taken as further evidence for the abundance of $p$-adic zeroes near~$1$.

($\cdots$)

\noindent
Finally, we remark that the possible existence of $p$-adic Siegel
zeroes and the failure of results such as the $p$-adic Brauer--Siegel
theorem indicate that it could be difficult, if not impossible, to do
analytic number theory with $p$-adic $L$-functions. For example, I do
not know how to obtain estimates on $\pi(x)$, the number of primes less
that or equal to $x$, using the fact that the $p$-adic zeta function has
a pole at $1$.}

\begin{remark}
One may explain what appens as follows, for simplicity in the case 
of a real quadratic field $K$ of character $\chi_K^{}$:

\smallskip
Roughly speaking, $v_p(L_p(1,\chi_K^{}))$ is closely related to 
$v_p(\order {\mathcal T}_K)$ and $v_p(L_p(0,\chi_K^{}))$ is closely related to 
$v_p(B_1(\omega^{-1}\,\chi_K^{}))$ ($\omega$ is the Teichm\"uller
character and $B_1(\omega^{-1}\,\chi_K^{})$ the generalized Bernoulli 
number of character $\omega^{-1}\,\chi_K^{}$), which is closely related to
the order of a suitable component of the $p$-class group
of the ``mirror field $K^*$'' (e.g., for $p=3$ and $K= \Q(\sqrt m)$, 
$K^*=  \Q(\sqrt {-3\,m})$); but since $\omega^{-1}\,\chi_K^{}$ is odd,
no unit intervenes and $v_p(L_p(0,\chi_K^{}))$ is usually ``small'' compared to
$v_p(\order {\mathcal T}_K)$ assumed to be ``very large'' (e.g.,
$m= 150094635296999122$ giving $v_3(\order {\mathcal T}_K)=19$
but $v_3(\order \Cl_{K^*})=1$). 
Thus, there exist in general ``Siegel zeros'' of $L_p(s,\chi_K^{})$, i.e., 
very close to $1$, which is an obstruction to a Brauer--Siegel strategy 
(see numerical illustrations for $p=2, 3$ in \cite{Wa3, Wa4, Wa5}).
\end{remark}

Consequently we will adopt another point of view.
Let $K \in {\mathcal K}_{\rm real}$ and let $p\geq 2$ be any fixed prime number.
As we have recalled it, $\order {\mathcal T}_K$ is in close 
relationship with $p$-adic $L$-functions (at $s=1$) of even 
Dirichlet characters in the abelian case (Kubota--Leopoldt, Barsky, Amice--Fresnel,...), 
or more generally with the residue at $s=1$ of the $p$-adic $\zeta$-function of 
$K$, built or study by many authors (Coates, Shintani, Barsky, 
Serre, Cassou-Nogu\`es, Deligne--Ribet, Katz, Colmez,...).
Conversely, there is no {\it algebraic} invariant (like a Galois group) interpreting the 
residue of the complex $\zeta$-function, but we have in this (archimedean) case numerous
inequalities.
So, we shall compare the complex and $p$-adic cases
to try to unify the set of all the points of view.
For this, we define normalizations of the $\zeta$-functions of a totally 
real number field (from \cite{Coa, Col}, then \cite{Gr4} for the regulators).

\subsection{Definitions and normalizations} \label{places}
Let $K \in {\mathcal K}_{\rm real}$ be of degree $d$ and let:
\ctl{${\mathcal P} := \{p_\infty, 2, 3, \ldots , p, \ldots\}$}
be the set of places of $\Q$, including the infinite place $p_\infty$ 
(we also use the symbol $\infty$ for real or complex 
functions, like ${\rm log}$-function, in the same logic 
as for $p$-adic ones; for instance, $R_{K,\infty}$
and $R_{K,p}$ shall be the usual regulators built with ${\rm log}_\infty$ and
${\rm log}_p$, respectively). We shall use, for any place $v \in {\mathcal P}$, 
subscripts ${(\bullet)}_{K,v}^{}$ for all invariants considered; when the context is clear, 
we omit $v$ ($p$-adic in most cases).

\subsubsection{\sc $v$-Cyclotomic extensions and $v$-conductors}
The $p$-cyclotomic $\Z_p$-exten\-sion is denoted $\Q^{{\rm c},p}$ and
we introduce $\Q^{{\rm c},p_\infty} := \Q$ as the ``${p_\infty}$-cyclotomic extension''.
We put $\Q^{{\rm c},v}=:\Q^{\rm c}$ for any $v \in {\mathcal P}$ 
if there is no ambiguity.
We attribute to the field $\Q$ the ``$v$-conductor'' ${\mathfrak f}^{}_{\Q, v} := p$ 
(resp. $4, 2$) if $v=p \ne 2$ (resp. $2, {p_\infty}$). We shall put $\sim$ for equalities 
up to a $p$-adic unit.

\subsubsection{\sc Normalized $\zeta$-functions at ${p_\infty}$}
We define at the infinite place ${p_\infty}$:
\begin{equation}\label{zeta}
\wt \zeta_{K,{p_\infty}}(s) := \Frac {{\mathfrak f}^{}_{K\cap \, \Q^{\rm c}}} 
{2^{d}} \cdot \zeta_{K, {p_\infty}}(s) = \Frac {1}{2^{d-1}} 
\cdot \zeta_{K, {p_\infty}}(s),\  s \in \C
\end{equation}
(see \cite[Remark III.2.6.5\,(ii)]{Gr1} for justifications about the factor 
$\frac {1}{2^d}$); then, let $h_K$ be the class number (ordinary sense), $R_{K,\infty}$ the 
classical regulator, $D_K$ the discriminant of $K$, and
${\mathcal W}_{K,{p_\infty}} := \plus_{w \mid p_\infty} \mu_{K_w}^{}\big / \mu_K^{}$,
of order $2^{d-1}$ since $K$ is totally real.
Then consider, with a perfect analogy with the $p$-adic case:
\begin{equation}\label{infini}
\order {\mathcal T}_{K,{p_\infty}} := 
h_K \cdot \Frac{R_{K,\infty}}{2^{d-1} \cdot \sqrt {D_K}} \cdot 
\order {\mathcal W}_{K,{p_\infty}} =
 h_K \cdot \Frac{R_{K,\infty}}{\sqrt {D_K}}.\ \footnote{\,The factor 
$\frac{R_{K,\infty}}{2^{d-1}\cdot \sqrt {D_K}}$ is by definition the normalized 
regulator ${\mathcal R}_{K,{p_\infty}}$ for $v={p_\infty}$, using the normalized 
log-function $\frac{1}{2} \cdot {\rm log}_\infty$ instead of ${\rm log}_\infty$; 
from \cite{AF}, it is defined without ambiguity. Thus, the factor $ {\mathcal W}_{K,{p_\infty}}$ 
does exist as in the $p$-adic case. The invariant ${\mathcal T}_{K,{p_\infty}}$
is related to the Arakelov class group of $K$ (see \cite{Sch} and its bibliography),
which gives the best interpretation.}
\end{equation}
Let $\wt\kappa_{K,{p_\infty}}^{}$ be the residue at $s=1$ of $\wt \zeta_{K,{p_\infty}}(s)$.
From the so-called complex ``analytic formula of the class number''  of $K$ 
(see, e.g., \cite[Chap. 4]{Wa1}), we get:
\begin{equation}\label{residu}
\wt \kappa_{K,{p_\infty}}^{} = h_K \cdot \Frac{R_{K,\infty}}{\sqrt {D_K}} 
= \order {\mathcal T}_{K,{p_\infty}}.
\end{equation}

\subsubsection{\sc Normalized $\zeta_p$-functions at $v=p$}
We define at a finite place $p$:
\begin{equation}\label{p}
\wt \zeta_{K,p}(s) := \Frac {{\mathfrak f}^{}_{K\cap \, \Q^{\rm c}}}
 {2^{d}} \cdot \zeta_{K,p}(s), \ s\in \Z_p,
\end{equation}
where ${\mathfrak f}^{}_{K\cap \, \Q^{\rm c}}$ is the conductor of 
$K\cap \, \Q^{\rm c}$ (if $K\cap \, \Q^{\rm c}$ is the $n$th stage in
$\Q^{\rm c}$, then ${\mathfrak f}^{}_{K\cap \, \Q^{\rm c}} \sim 
2\,p\cdot [K\cap \, \Q^{\rm c} : \Q] \sim 2\,p^{n+1}$); since from  \cite{Coa, Col, Se},
the residue of $ \zeta_{K,p}(s)$ at $s=1$ is
$\kappa_{K,p}^{} \sim \Frac{2^{d-1}\cdot h_K\cdot R_{K,p}}{\sqrt {D_K}}$,
we get the normalized $p$-adic residue:
\begin{equation}\label{fini}\hspace{0.6cm}
\wt\kappa_{K,p}^{} =  \Frac {{\mathfrak f}^{}_{K\cap \, \Q^{\rm c}}} 
{2^{d}} \cdot \kappa_{K,p}^{} \sim \order {\mathcal T}_{K,p} \ 
\hbox{(see Subsection \ref{abelfunct} for the abelian case)}.
\end{equation}

\noindent
So, the residues of the normalized $\zeta_v$-functions of $K$ 
are, for all $v \in {\mathcal P}$, such that:
\ctl{$\wt\kappa_{K,v}^{} := \ds \lim_{s \to 1} (s-1)\cdot
\Frac {{\mathfrak f}^{}_{K\cap \, \Q^{\rm c}}} {2^{d}} \cdot \zeta_{K,v}(s)
\sim \order {\mathcal T}_{K,v}$,}
which is the order of an arithmetical invariant for finite places $v=p$ and 
the measure of a real volume for $v={p_\infty}$ (see the last footnote).

\subsection{Abelian complex $L$-functions -- Upper bounds}
In the abelian case:
\begin{equation}\label{analytic}
\order {\mathcal T}_{K,{p_\infty}} = h_K \cdot \Frac{R_{K,\infty}}{\sqrt {D_K}} = 
\prd_{\chi \ne 1} \, \Frac{1}{2}\,L_{p_\infty} (1,\chi),
\end{equation}
where $\chi$ goes through all the corresponding Dirichlet 
characters of $K$ with conductor $f_\chi$, and where
$L_{p_\infty}$ denotes the complex $L$-function. 
If $K=\Q(\sqrt m)$, of fundamental unit $\varepsilon_K^{}$
and quadratic character $\chi_K^{}$, one gets:
\begin{equation*}
\order {\mathcal T}_{K,{p_\infty}} = 
h_K \cdot \Frac{{\rm log}_\infty(\varepsilon_K^{}) }{\sqrt {D_K}}
= \Frac{1}{2} \cdot L_{p_\infty}(1, \chi_K^{}).
\end{equation*} 

For each $L_{p_\infty}(1, \chi)$ one has many upper bounds which 
are improvements of the classical inequality
$\Frac{1}{2} \cdot L_{p_\infty}(1, \chi) \leq 
\big(1 +o(1)\big) \cdot {\rm log}_\infty(\sqrt{f_\chi})$.
In \cite[Corollaire 1]{Ra} one has, for even primitive characters:
\begin{equation*}
\hbox{$\Frac{1}{2} \cdot L_{p_\infty}(1, \chi) \leq 
\Frac{1}{2} \cdot {\rm log}_\infty(\sqrt {f_\chi})$,}
\end{equation*}
giving from the previous definition \eqref{infini} and formula \eqref{analytic}:
\begin{equation}\label{major}
{\rm log}_\infty(\order {\mathcal T}_{K,{p_\infty}}) 
\leq {\mathcal C}_{p_\infty} \cdot {\rm log}_\infty(\sqrt {D_K}),
\end{equation}
with an explicit constant ${\mathcal C}_{p_\infty}$ if $K$ runs trough the set
of real abelian fields such that $\Frac{d}{{\rm log}_\infty(\sqrt {D_K})} \to 0$,
for instance in the simplest form of Brauer--Siegel theorem.

\medskip
We shall give numerical complements in Subsection \ref{avsp} by means of
computations of lower and upper bounds of:
\begin{equation*}\label{majorinfini}
C_{p_\infty}(K)=\wt {B\!S}_K :=
\Frac{{\rm log}_\infty(\order {\mathcal T}_{K,{p_\infty}})}{{\rm log}_\infty(\sqrt{D_K})}
\end{equation*}
(see Definition \ref{BSTV}).

\begin{remark}\label{note1}
For the sequel, we do not need any sophisticated upper bound
(only the existence of ${\mathcal C}_{p_\infty}$), but one may refer to
\cite{GS,L2,L3,Pin,Ra} for other inequalities; for instance,
one gets, for real abelian fields $K$ of degree $d$, with our notations:
\ctlm{$\order {\mathcal T}_{K,{p_\infty}} := h_K \cdot \Frac{R_{K,\infty}}{\sqrt {D_K}} \leq 
\Big( \Frac{1}{2}\,\Frac{{\rm log}_\infty(\sqrt{D_K})}{d-1} \Big)^{d-1}$,}
thus in the cases $d=2$ and $d=3$:
$$\order {\mathcal T}_{K,{p_\infty}} \!\! = h_K \cdot 
\Frac{{\rm log}_\infty(\varepsilon_K^{}) }{\sqrt {D_K}} \leq 
\Frac{1}{2}\,{\rm log}_\infty(\sqrt{D_K}),\ 
\order {\mathcal T}_{K,{p_\infty}} \!\! = h_K \cdot 
\Frac{R_{K,\infty}}{\sqrt {D_K}} \leq 
\Frac{1}{16} \Big({\rm log}_\infty(\sqrt{D_K}) \Big)^{\!2} \!,$$
respectively. In the quadratic and cubic cases one shows that:
\begin{equation}\label{lou}
h_K \leq \Frac{1}{2} \cdot \sqrt{D_K}, \hspace{0.5cm}
h_K \leq \Frac{2}{3} \cdot \sqrt{D_K},\ \hbox{ respectively}.
\end{equation}
\end{remark}

\subsection{Abelian $L_p$-functions}\label{abelfunct}
The Kubota--Leopoldt $p$-adic $L$-functions give rise to the
analytic formula \cite[\S\,2.1 \& Th\'eor\`eme 6, \S\,2.3]{AF}:
\ctl{$h_K \cdot \Frac{R_{K,p}}{\sqrt {D_K}}
\sim \prd_{\chi \ne 1} \,  \Frac{1}{2}\,L_p (1,\chi)
\cdot \prd_{\chi \ne 1} \,\Big(1-\Frac{\chi (p)}{p} \Big)^{-1}$.}
The ``$p$-adic class number formula'' 
for real abelian fields uses the formula of \cite{Coa}:
\ctl{$\order {\mathcal T}_{K,p} \sim [K \cap \, \Q^{\rm c} : \Q] \cdot
 \Frac{p}{\prod_{{\mathfrak p}\mid p} {\rm N}{\mathfrak p}}
\cdot h_K \cdot \Frac{R_{K,p}} {\sqrt {D_K}}$.}
Thus, since $\prod_{\chi} \,\Big(1-\Frac{\chi (p)}{p} \Big)^{-1} 
=\prod_{{\mathfrak p}\mid p} (1- {\rm N}{\mathfrak p}^{-1})^{-1} \sim 
\prod_{{\mathfrak p}\mid p} {\rm N}{\mathfrak p}$, this yields:
\begin{equation}\label{59}
\order {\mathcal T}_{K,p} \sim
 \Frac{p}{\prod_{{\mathfrak p}\mid p} {\rm N}{\mathfrak p}}
\!\cdot\! \prd_{\chi \ne 1} \Frac{1}{2}\,L_p (1,\chi)
\!\cdot\!\! \prd_{\chi \ne 1} \Big(1-\Frac{\chi (p)}{p} \Big)^{-1} \!\!\!
\sim \prd_{\chi \ne 1} \Frac{1}{2}\,L_p (1,\chi) = \wt\kappa_{K,p}^{}.
\end{equation}

But no upper bound of the $p$-adic valuation of this residue is known. 
So we must, on the contrary, try to study directely $\order {\mathcal T}_{K,p}$ 
with arithmetic tools.

\subsection{Arithmetical study of $\wt \kappa_{K,p}^{}$} To study this residue, 
consider \eqref{59} giving $\wt\kappa_{K,p}^{} \sim \order {\mathcal T}_{K,p}$.
In $\order {\mathcal T}_K = 
\order \Cl_K^{\,\rm c} \cdot \order {\mathcal R}_K \cdot \order{\mathcal W}_K$,
the computation of $\order {\mathcal W}_K$ is obvious.
Then $\order \Cl_K^{\,\rm c} = \Frac{\order \Cl_K}{[H_K \cap K^{\rm c} : K]}
= \order \Cl_K \cdot \frac{1}{e_p}\cdot 
(\langle -1 \rangle : \langle -1 \rangle \cap  {\rm N}_{K/\Q}(U_K)) \cdot [K \cap \Q^{\rm c} : \Q]$,
where $e_p$ is the ramification index of $p$ in $K/\Q$ \cite[Theorem III.2.6.4]{Gr1}. 
So, for $p \gg 0$ we get $\order \Cl_K^{\,\rm c} \cdot \order {\mathcal W}_K = 1$.
Then the main factor is (whatever the field $K$ and the prime $p$
\cite[Proposition 5.2]{Gr4}):
\begin{equation}\label{regulator}
\order {\mathcal R}_K = \order {\rm tor}_{\Z_p}^{}\big({\rm log}_p\big (U_K \big)  \big / 
{\rm log}_p (\ov E_K) \big) \sim \Frac{1}{2} \cdot
\Frac{\big(\Z_p : {\rm log}_p({\rm N}_{K/\Q}(U_K)) \big)}
{ \order {\mathcal W}_K \cdot \prod_{{\mathfrak p} \mid p}{\rm N} {\mathfrak p}}
\cdot \Frac {R_{K,p}}{\sqrt {D_K}},
\end{equation}
which is unpredictible and more complicate 
if $p$ ramifies in $K$ or if $p=2$.

\smallskip\noindent
 In the non-ramified case for $p\ne 2$, 
it is given by the classical detrminant provided that
one replaces ${\rm log}_p$ by the ``normalized logarithm'' $\frac{1}{p}\,{\rm log}_p$.

\begin{remarks} \label{defdelta} Let $K=\Q(\sqrt m)$ and let $p \nmid D_K$
with residue degree $f \in \{1,2\}$. 

(i) For $p\ne 2$, 
$\order {\mathcal R}_K \sim \frac{1}{p}\,{\rm log}_p(\varepsilon_K^{}) \sim 
p^{\delta_p(\varepsilon_K^{})}$, where $\delta_p(\varepsilon_K^{}) = 
v_p \Big (\Frac{\varepsilon_K^{p^f-1} - 1}{p} \Big)$. 

(ii) For $p=2$, the good definition of the $\delta_2$-function is
$\delta_2(\varepsilon_K^{}):=
v_2 \Big (\Frac{\varepsilon_K^{2} - 1}{8} \Big)$ if $f=1$ and
$v_2 \Big (\Frac{\varepsilon_K^{6} - 1}{4} \Big)$
if $f=2$, in which cases $\order {\mathcal R}_K \sim 2^{\delta_2(\varepsilon_K)}$.

\smallskip
(iii) The existence of an upper bound for $v_p(\frac{1}{2}L_p(1, \chi_K^{}))$ would 
be equivalent to an estimation of the order of magnitude of ${\delta_p(\eta_K^{})}$ for
the cyclotomic number $\eta_K^{} := \prod_{a,\chi(a)=1} (1 - \zeta_{D_K}^a)$,
where $\zeta_{D_K}$ is a primitive $D_K$th root of unity
(interpretation of the class number formula via cyclotomic units).
The study given in \cite[Th\'eor\`eme 1.1]{Gr2}, and applied to the number
$\xi = 1 - \zeta_{D_K}$, suggests that if $p\to\infty$, the probability of 
$\delta_p(\eta_K^{}) \geq 1$ for the $\chi_K^{}$-component 
$\langle\, \eta_K^{} \, \rangle_\Z = \langle\,  \xi \, \rangle^{e_{\chi_K}}$, of the 
Galois module generated by $\xi$, tends to $0$ at least as $O(1)\cdot p^{-1}$ 
and conjecturaly as $p^{-({\rm log}({\rm log}(p))/{\rm log}(c_0(\eta_K^{}))-O(1))}$, 
where $c_0(\eta_K^{}) = \vert \eta_K^{} \vert >1$; this does not apply to small~$p$.
This explains the specific difficulties of the $p$-adic case,
which is not surprising since the study of $v_p(\order {\mathcal T}_K)$ 
represents a refinement of Leopoldt's conjecture.
\end{remarks}

We intend to give estimations of $v_p(\order {\mathcal T}_K)$  ($p$ fixed)
related to the discriminant $D_K$ when $K$ varies in a family 
${\mathcal K} \subseteq {\mathcal K}_{\rm real}$  (as in \cite{TV}, we call 
{\it family of number fields} any infinite set of non-isomorphic 
number fields $K$; thus, the condition $D_K \to \infty$ makes sense in 
${\mathcal K}$). In a numerical point of view, we 
shall analyse the set ${\mathcal K}_{\rm real}^{(2)}$ of real quadratic fields 
and the subset ${\mathcal K}_{\rm ab}^{(3)}$ of ${\mathcal K}_{\rm real}^{(3)}$ 
(totally real cubic fields), of cyclic cubic fields of conductor $f$, 
described by the polynomials (see, e.g.,~\cite{ET}):
\begin{equation}\label{polcubic}
\begin{aligned}
P & = X^3+X^2 - \Frac{f -1}{3} \cdot X + \Frac{1+f \, (a-3)}{27}, \  \hbox{if $3 \nmid f$}, \\
P & = X^3-\Frac{f}{3} \cdot X -\Frac{f \,a}{27},\ \  \hbox{if $3 \mid f$},
\end{aligned}
\end{equation}
where $f=\Frac{a^2+27\,b^2}{4}$ with 
$a \equiv 2 \pmod 3$ (if $3 \nmid f$), $a \equiv 6 \pmod 9$ $\&$
$b \not \equiv 0 \pmod 3$ (if $3 \mid f$). Some non-cyclic cubic
fields will also be considered.

\smallskip
In the forthcoming Sections, we deal only with finite places $p$; 
so we simplify some notation in an obvious way.

\section{Direct calculation of $v_p(\order {\mathcal T}_K)$ via PARI/GP}\label{section3}

The programs shall try to verify a $p$-adic analogue of the 
relation \eqref{major}, for quadratic and cubic fields; for each fixed $p$,
they shall give the successive minima of the expression
$\Delta_p(K) := \Frac{{\rm log}_\infty(\sqrt{D_K\,})}{{\rm log}_\infty(p)} - 
v_p(\order {\mathcal T}_K)$ and the successive maxima of:
\begin{equation}\label{cpk}
C_p(K) := \Frac{v_p(\order {\mathcal T}_K) \cdot {\rm log}_\infty(p)}
{{\rm log}_\infty(\sqrt{D_K})},
\end{equation}
when $D_K$ increases in the selected family ${\mathcal K}$.
It seems that a first minimum of $\Delta_p(K)$ 
(on an interval $I$ for $D_K$) is rapidely obtained and is
negative of small absolute value, giving $C_p(K)>1$; 
whence the interest of the computation 
of $C_p(K)$ and the question of the existence of 
${\mathcal C}_p = \sup_{K \in {\mathcal K}} (C_p(K))$.
If ${\mathcal C}_p = \infty$, this means that (for example)
$v_p(\order {\mathcal T}_{K_i}) = {\rm log}_\infty(\sqrt{D_{K_i}}) \cdot 
O\big({\rm log}_\infty({\rm log}_\infty(\sqrt{D_{K_i}}))\big)$
for infinitely many $K_i \in {\mathcal K}$, whence, in our opinion,
the ``excessive relations'' $\order {\mathcal T}_{K_i} \gg \sqrt{D_{K_i}}$.

\smallskip
We shall observe that $\sup_{D <  x}(C_p(K))$ increases
and stabilizes rapidely, for a rather small $D_0$; this means that 
$C_p(K)$ is locally decreasing for $D_K \gg D_0$, whence the interest of 
calculating $C_p(K)$ for discriminants as large as possible to expect the 
existence of $\limsup_{K \in {\mathcal K}} (C_p(K))$ of a different nature 
(see the very instructive example discussed in the \S\,\ref{vph}\,(i)).

\smallskip
We shall adapt the following PARI program \cite[\S\,3.2]{Gr5} (testing the 
$p$-rationality of {\it any number field $K$}), that we recall 
for the convenience of the reader (for this, choose any monic 
irreducible polynomial $P$ and any prime $p$; 
the program gives in $S$ the signature $(r_1, r_2)$
of $K$, then $r := r_2+1$; recall that from $K=bnfinit(P,1)$, one gets
$D_K=component(component(K,7),3)$ and that
from $C8=component(K,8)$, the structure of the class group, 
the regulator and a fundamental system of units are given by 
$component(C8,1)$, $component(C8,2)$, and $component(C8,5)$, 
respectively; whence the class number given by 
$h_K=component(component(C8,1),1)$):

\footnotesize\smallskip
\begin{verbatim}
{P=x^6-123*x^2+1;p=3;K=bnfinit(P,1);n=2;if(p==2,n=3);Kpn=bnrinit(K,p^n);
S=component(component(Kpn,1),7);r=component(component(S,2),2)+1;
print(p,"-rank of the compositum of the Z_",p,"-extensions: ",r);
Hpn=component(component(Kpn,5),2);L=listcreate;e=component(matsize(Hpn),2);
R=0;for(k=1,e,c=component(Hpn,e-k+1);if(Mod(c,p)==0,R=R+1;
listinsert(L,p^valuation(c,p),1)));
print("Structure of the ",p,"-ray class group:",L);
if(R>r,print("rk(T)=",R-r," K is not ",p,"-rational"));
if(R==r,print("rk(T)=",0," K is ",p,"-rational"))}

3-rank of the compositum of the Z_3-extensions: 2
Structure of the 3-ray class group: List([9, 9, 9])
rk(T)=1  K is not 3-rational
\end{verbatim}
\normalsize

\noindent
For any $K \in {\mathcal K}_{\rm real}$, the $p$-invariants of ${\rm Gal}(K(p^n)/ K)$, 
where $K(p^n)$ is the ray class field of modulus $(p^n)$ for any $n \geq 0$, 
are given by the following simplest program (in which $n=0$ gives the 
structure of the $p$-class group):

\footnotesize
\begin{verbatim}
{P=x^2-2*3*5*7*11*13*17;K=bnfinit(P,1);p=2;n=18;Kpn=bnrinit(K,p^n);
Hpn=component(component(Kpn,5),2);L=listcreate;e=component(matsize(Hpn),2);
for(k=1,e,c=component(Hpn,e-k+1);if(Mod(c,p)==0,
listinsert(L,p^valuation(c,p),1)));print(L)}
List([131072, 2, 2, 2, 2, 2])
\end{verbatim}
\normalsize

\smallskip
For $n=0$ one gets $\Cl_K \simeq [2, 2, 2, 2, 2]$.
Taking $n$ large enough in the program allows us to compute 
directely the structure of ${\mathcal T}_K$ as is done by a 
precise (but longer) program in \cite{PV}. This gives the $p$-valuation 
in $vptor$ of $\order {\mathcal T}_K$ as rapidely as possible; 
for this, explain some details about PARI (from~\cite{P}).

\smallskip
Let $K \in {\mathcal K}_{\rm real}$ be linearly disjoint from $\Q^{\rm c}$;
let $K(p^n)$ be the ray class field of modulus $(p^n)$,
$n \geq 2$ (resp. $n \geq 3$) if $p\ne 2$ (resp. $p=2$); indeed,
from \cite[Theorem 2.1]{Gr5}, these conditions on $n$
are sufficient to give the $p$-rank $t_K =: t$ of~${\mathcal T}_K$.
Thus,  for $n$ large enough, the $p$-structure of ${\rm Gal}(K(p^n)/ K)$
is of the form $[p^a, p^{a_1}, \ldots ,  p^{a_t}]$, with
$a \geq a_1 \geq \cdots \geq a_t$, in 
$Hpn := component(component(Kpn,5),2)$, where $Kpn=bnrinit(K,p^n)$
and $p^a = [K(p^n) \cap K^{\rm c} : K]$. 

\smallskip
Then $\order {\mathcal T}_K = [K(p^n) : K] \times p^{-a}$
(up to a $p$-adic unit), where $p^a$ is the largest component given in $Hpn$
(whence the first one in the list, under the condition 
$n \gg \max(a_1, \ldots, a_t)$); so we have only to verify
that $p^n$ is much larger than the exponent $\max(p^{a_1}, \ldots, p^{a_t})$
of ${\mathcal T}_K$.

\smallskip
In practice, and to obtain fast programs, we must look at the order of
magnitude of the results to increase $n$ if necessary; in fact, once
the part $K=bnfinit(P,1)$ of the program is completed, a large value of $n$
does not significantly increase the execution time. 
For instance, with $P=x^2-4194305$ and $p=2$, one gets the 
successive structures for $2 \leq n \leq 16$:

\footnotesize
\begin{verbatim}
2  [2, 2]       6  [32, 16, 2]    10  [512, 256, 2]     14  [8192, 2048, 2]
3  [4, 2, 2]    7  [64, 32, 2]    11  [1024, 512, 2]    15  [16384, 2048, 2] 
4  [8, 4, 2]    8  [128, 64, 2]   12  [2048, 1024, 2]   16  [32768, 2048, 2]
5  [16, 8, 2]   9  [256, 128, 2]  13  [4096, 2048, 2]
\end{verbatim}
\normalsize

\noindent
showing that $n$ must be at least $13$ to give
${\mathcal T}_K \simeq  \Z/2^{11}\Z \times \Z/2\Z$.
In the forthcomming numerical results, if any doubt occurs for a specific field, 
it is sufficient to use the previous program with bigger $n$.

\section{Numerical investigations for real quadratic fields}

Let $K=\Q(\sqrt m)$, $m>0$ squarefree.
We have $\order {\mathcal W}_K = 2$ for $p=2 \ \& \  m \equiv \pm 1 \pmod 8$, 
$\order {\mathcal W}_K =3$ for $p=3 \ \& \ m \equiv -3 \pmod 9$,
and we are mainely concerned with the $p$-class group $\Cl_K$ and the 
normalized regulator~${\mathcal R}_K$.
When $p>2$ is unramified, we have $v_p(\order {\mathcal R}_K) = \delta_p(\varepsilon)$ 
for the fundamental unit $\varepsilon$ of $K$ and if $p=2$ is unramified, we have 
$\delta_2(\varepsilon):= v_2 \Big (\Frac{\varepsilon^{2\cdot(2^f-1)} - 1}{2^{4-f}} \Big)$ 
where $f$ is the residue degree of $2$ in $K$
(see Remarks \ref{defdelta} (i), (ii)).
So, we may compute $v_p(\order {\mathcal T}_K)$ as 
$v_p(\order \Cl_K^{\rm c})+ \delta_p(\varepsilon)+ v_p(\order {\mathcal W}_K)$ and
we shall compare with the direct computation of the structure of ${\mathcal T}_K$
as explain above. Remark that, for $p=2$, $\order \Cl_K = 2\cdot \order \Cl_K^{\rm c}$
(instead of $\order \Cl_K^{\rm c}$) if and only if $m \equiv 2 \pmod 8$, in which case 
$H_K \cap K^{\rm c}=K(\sqrt 2)$ is unramified over $K$.

\smallskip
We have the following result, about $v_p(\order {\mathcal R}_K)$,
when $p \geq 2$ ramifies:

\begin{proposition}\label{casramif}
For $K=\Q(\sqrt m)$ real and $p \mid D_K$, $v_p(\order {\mathcal R}_K)$
is given as follows:

\smallskip
(i) For $p \nmid 6$ ramified, $\order {\mathcal R}_K  \sim
\frac{1}{\sqrt m} \cdot {\rm log}_p(\varepsilon)$ and
$v_p(\order {\mathcal R}_K) = \delta$ if
$v_{\mathfrak p}(\varepsilon^{p-1} - 1) = 1+2\,\delta$, where ${\mathfrak p} \mid p$,
$\delta \geq 0$. 

\smallskip
(ii) For $p=3$ ramified, $\order {\mathcal R}_K \sim
\frac{1}{\sqrt m} \cdot {\rm log}_3(\varepsilon)$ (resp.
$\order {\mathcal R}_K \sim \frac{1}{3\,\sqrt m} \cdot {\rm log}_3(\varepsilon)$)
if $m\not\equiv -3 \pmod 9$ (resp. $m \equiv -3 \pmod 9$).
Then  $v_3(\order {\mathcal R}_K)=(v_{\mathfrak p} (\varepsilon^6 - 1) -2-\delta)/2$
where ${\mathfrak p} \mid 3$ and $\delta=1$ (resp. $\delta=3$) 
if $m\not\equiv -3 \pmod 9$ (resp. $m \equiv -3 \pmod 9$).

\smallskip
(iii) For $p=2$ ramified, $\order {\mathcal R}_K \sim 
\frac{{\rm log}_2(\varepsilon)}{2\,\sqrt m}$
(resp. $\frac{{\rm log}_2(\varepsilon)}{4}$) if $m \not \equiv -1 \pmod 8$
(resp. $m \equiv -1 \pmod 8$). Then,  $v_2(\order {\mathcal R}_K)=
(v_{\mathfrak p} (\varepsilon^4 - 1) - 4 -\delta)/2$, 
where ${\mathfrak p} \mid 2$ and where $\delta=1, 2, 3, 4$
if $m \equiv 2, 3, 6, 7 \pmod 8$, respectively).
\end{proposition}

\begin{proof} Exercise using the expression \eqref{regulator} of 
$\order {\mathcal R}_K$ where ${\rm N}_{K/\Q}(U_K)$ is of index $2$
in $U_\Q$ (local class field theory), the fact that ${\rm N}_{K/\Q}(\varepsilon) = \pm 1$
(i.e., ${\rm Tr}_{K/\Q}({\rm log}_p(\varepsilon)) = 0$),
and the classical computation of a $p$-adic logarithm.
\end{proof}

\begin{remark}
A first information is then the order of magnitude of $\delta_p(\varepsilon)$
as $D_K \to \infty$ ($p$ fixed). Its non-nullity for $p \gg 0$ ($K$ fixed) is 
a deep problem for which we can only give some numerical experiments.
For $p \gg 0$ and any $K \in {\mathcal K}_{\rm real}$, an extensive 
schedule is discussed in \cite{Gr2}, for the
study of $p$-adic regulators of an algebraic number
$\eta \in K^\times$ (giving ``Frobenius determinants''), whose properties are 
characterized by the Galois $\Z_p$-module generated 
by its ``Fermat quotient'' $\frac{1}{p} \,(\eta^{p^f-1}-1)$. 

\smallskip
These questions, applied in our study to 
a ``Minkowski unit'', are probably the explanation of the failure 
of the classical $p$-adic analysis of $\zeta_p$-functions (among many other 
subjects in number theory) since such Fermat quotients problems are neither easier nor more 
difficult than, for instance, the famous problem of Fermat quotients of the 
number $2$, for which no one is able to say, so far, how much $p$ are 
such that $\frac{1}{p}(2^{p-1}-1) \equiv 0 \pmod p$.
\end{remark}

\subsection{Maximal values of $v_p(\order {\mathcal R}_K)$}\label{max}
Consider a prime $p$ fixed and the family ${\mathcal K}_{\rm real}^{(2)}$.
 The following programs find the 
successive maxima of $\delta_p(\varepsilon)$ with the corresponding
increasing $D_K \in [bD, BD]$; the programs use the fact that for 
$p$ unramified, in the inert case, $\varepsilon^{p+1} \equiv 
{\rm N}_{K/\Q}(\varepsilon) \pmod p$, otherwise, 
$\varepsilon^{p-1} \equiv 1 \pmod p$. 

\smallskip
We shall indicate if necessary the maximal value obtained for $C_p(K)$ 
defined by the expression \eqref{cpk} by computing $v_p(\order {\mathcal T}_K)=
\delta_p(\varepsilon) + v_p(\order \Cl_K^{\,\rm c})+v_p(\order {\mathcal W}_K)$.

\subsubsection {\sc Program for $p = 2$ unramified}\label{p=2}
For $p=2$ unramified, we use the particular formula given 
in Remark \ref{defdelta}\,(ii).

\footnotesize
\begin{verbatim}
{bD=5;BD=5*10^7;Max=0;for(D=bD,BD,if(core(D)!=D,next);ss=Mod(D,8);s=0;
if(ss==1,s=1);if(ss==5,s=-1);if(s==0,next);E=quadunit(D)^2;A=(E^(2-s)-1)/(2*s+6);
A=[component(A,2),component(A,3)];delta=valuation(A,2);
if(delta>Max,Max=delta;print("D=",D," delta=",delta)))}
D=21   delta=1   D=1185   delta=8    D=115005  delta=13   D=1051385   delta=19
D=41   delta=3   D=1201   delta=10   D=122321  delta=14   D=12256653  delta=21
D=469  delta=5   D=3881   delta=11   D=222181  delta=16   D=14098537  delta=22 
D=645  delta=6   D=69973  delta=12   D=528077  delta=18   D=28527281  delta=25
\end{verbatim}
\normalsize

\noindent
The next discriminant in $[5\cdot 10^7, 5 \cdot 10^8]$ (two days of computer) 
is $D_K= 214203013$, where $\delta_2(\varepsilon) = 26$, $v_2(h_K)=1$, 
$v_2(\order {\mathcal W}_K)=0$, $v_2(\order{\mathcal T}_K)=27$, $C_2(K)=1.951261$.

\subsubsection {\sc Program for $p = 2$ ramified}
A similar program using Proposition \ref{casramif}(iii) gives analogous 
results for maximal values of $\delta_2(\varepsilon)$:

\footnotesize
\begin{verbatim}
{bm=3;Bm=5*10^7;Max=0;for(m=bm,Bm,s=Mod(m,4);ss=Mod(m,8);
if(core(m)!=m || s==1,next);A=(quadunit(4*m)^4-1)/4;N=norm(A);v=valuation(N,2);
if(s==2,delta=v-3);if(ss==3,delta=v-2);if(ss==7,delta=v-4);delta=delta/2;
if(delta>Max,Max=delta;print("D=",4*m," delta=",delta)))}
D=28   delta=1   D=508    delta=6   D=28664    delta=13   D=15704072  delta=21         
D=124  delta=2   D=1784   delta=7   D=81624    delta=17   D=29419592  delta=22     
D=264  delta=3   D=10232  delta=8   D=1476668  delta=18   D=36650172  delta=23      
D=456  delta=5   D=21980  delta=9   D=2692776  delta=19   D=80882380  delta=28                  
\end{verbatim} 
\normalsize 

\noindent
For $D_K= 80882380=4 \cdot 5 \cdot 239 \cdot 16921$, $\delta_2(\varepsilon) = 28$, 
$v_2(h_K)=2$, $v_2(\order{\mathcal W}_K)=0$, $v_2(\order{\mathcal T}_K)=30$,
$C_2(K) = 2.2840$, whence the influence of genera theory on $C_2(K)$.

\subsubsection {\sc Program for any unramified $p \geq 3$} 
The program can be simplified:

\footnotesize
\begin{verbatim}
{p=3;bD=5;BD=10^8;Max=0;for(D=bD,BD,e=valuation(D,2);M=D/2^e;if(core(M)!=M,next);
if((e==1||e>3)||(e==0 & Mod(M,4)!=1)||(e==2 & Mod(M,4)==1),next);s=kronecker(D,p);
if(s==0,next);E=quadunit(D);nu=norm(E);u=(1+nu-nu*s+s)/2;
A=(E^(p-s)-u)/p;A=[component(A,2),component(A,3)];delta=valuation(A,p);
if(delta>Max,Max=delta;print("D=",D," delta=",delta)))}
D=29   delta=2    D=13861  delta=7    D=321253   delta=12   D=21242636  delta=16  
D=488  delta=4    D=21713  delta=9    D=6917324  delta=13   D=71801701  delta=19 
D=1213 delta=6    D=153685 delta=10   D=13495160 delta=14      
\end{verbatim}
\normalsize

\smallskip\noindent
which gives $\delta_3(\varepsilon) \leq 19$ on the interval $[2, 10^8]$, 
obtained for $D_K = 71801701$, where 
$v_3(h_K)=v_3(\order{\mathcal W}_K)=0$, $v_3(\order{\mathcal T}_K)=19$, 
$C_3(K)=2.307828$.

\subsubsection {\sc Programs for $p = 3$ ramified} We obtain (cf. Proposition \ref{casramif}\,(ii)):

\footnotesize
\begin{verbatim}
{bD=5;BD=10^8;Max=0;for(D=bD,BD,e=valuation(D,2);M=D/2^e;if(core(M)!=M,next);
if(Mod(M,3)!=0||(e==1||e>3)||(e==0 & Mod(M,4)!=1)||(e==2 & Mod(M,4)==1),next);
E=quadunit(D)^6;A=norm(E-1);v=valuation(A,3);if(Mod(D,9)!=-3,delta=(v-3)/2);
if(Mod(D,9)==-3,delta=(v-5)/2);if(delta>Max,Max=delta;
print("D=",D," delta=",delta)))}
D=93   delta=1   D=1896    delta=6   D=2354577  delta=11   D=104326449  delta=15
D=105  delta=2   D=102984  delta=8   D=6099477  delta=12   D=448287465  delta=18
D=492  delta=3   D=168009  delta=10  D=17157729 delta=13
\end{verbatim}
\normalsize

\subsubsection {\sc Program for any ramified $p>3$} 
Let's illustrate this case with a large~$p$:

\footnotesize
\begin{verbatim}
{p=1009;bD=5;BD=10^8;Max=0;for(D=bD,BD,e=valuation(D,2);M=D/2^e;if(core(M)!=M,next);
if(Mod(M,p)!=0||(e==1||e>3)||(e==0 & Mod(M,4)!=1)||(e==2 & Mod(M,4)==1),next);
E=quadunit(D)^(p-1);A=norm(E-1);delta=(valuation(A,p)-1)/2;
if(delta>Max,Max=delta;print("D=",D," delta=",delta)))}
D=1900956   delta=1
\end{verbatim}
\normalsize

\smallskip\noindent
For large $p$ (ramified or not) there are few solutions in a reasonable interval since
we have, roughly speaking, 
${\rm Prob}\big (\delta_{p}(\varepsilon) \geq \delta \big) \approx p^{-\delta}$, 
otherwise, the solutions are often with $\delta_{p}(\varepsilon)=1$, large $D_K$, 
$C_p(K)$ being rather small as we shall analyse now.

\subsection{Experiments for a conjectural upper bound - Quadratic fields}\label{numexp}
We only assume $K \ne \Q(\sqrt 2)$ when $p=2$ to always have $K \cap \Q^{\rm c} = \Q$. 
We have given previously programs for the maximal values of $v_p(\order {\mathcal R}_K)$; 
we now give the behaviour of the whole $v_p(\order {\mathcal T}_K)$ for
increasing discriminants;  for this purpose, we compute:

\smallskip
\centerline{$\Delta_p(K) := 
\Frac{{\rm log}_\infty(\sqrt{D_K})}{{\rm log}_\infty(p)} - v_p(\order {\mathcal T}_K)\ $ and 
$\ C_p(K) := \Frac{v_p(\order {\mathcal T}_K) \cdot {\rm log}_\infty(p)}
{{\rm log}_\infty(\sqrt{D_K})}$.}

\subsubsection{\sc Program for $p=2$}\label{v2tor}
The numerical data are $D_K$, $v_p(\order {\mathcal T}_K)$ (in $vptor$; for this
choose $n$ large enough), the successive $\Delta_p(K)$ (in $Y\!min$) and the 
corresponding $C_p(K)$ (in $Cp$); we omit the $2$-rational fields (for them, $vptor=0$):

\footnotesize
\begin{verbatim}
{p=2;n=36;bD=5;BD=10^6;ymin=5;for(D=bD,BD,e=valuation(D,2);M=D/2^e;
if(core(M)!=M,next);if((e==1||e>3)||(e==0&Mod(M,4)!=1)||(e==2 & Mod(M,4)==1),next);
m=D;if(e!=0,m=D/4);P=x^2-m;K=bnfinit(P,1);Kpn=bnrinit(K,p^n);C5=component(Kpn,5);
Hpn0=component(C5,1);Hpn=component(C5,2);Hpn1=component(Hpn,1);
vptor=valuation(Hpn0/Hpn1,p);Y=log(sqrt(D))/log(p)-vptor;
if(Y<ymin,ymin=Y;Cp=vptor*log(p)/log(sqrt(D));
print("D=",D," m=",m," vptor=",vptor," Ymin=",Y," Cp=",Cp)))}
D=17      m=17        vptor=1      Ymin=1.04373142...     Cp=0.4893
D=28      m=7         vptor=2      Ymin=0.40367746...     Cp=0.8320
D=41      m=41        vptor=4     Ymin=-1.32122399...     Cp=1.4932
D=508     m=127       vptor=7     Ymin=-2.50565765...     Cp=1.5575
D=1185    m=1185      vptor=10    Ymin=-4.89466432...     Cp=1.9587
D=1201    m=1201      vptor=11    Ymin=-5.88498978...     Cp=2.1505
D=3881    m=3881      vptor=12    Ymin=-6.03889364...     Cp=2.0130
D=11985   m=11985     vptor=13    Ymin=-6.22552885...     Cp=1.9189
D=26377   m=26377     vptor=14    Ymin=-6.65650356...     Cp=1.9064
D=81624   m=20406     vptor=20   Ymin=-11.84164710...     Cp=2.4514
\end{verbatim}
\normalsize

The larger computations in \S\,\ref{p=2} show the largest case
$D_K=214203013$ with $h_K=2$ and $\delta_2(\varepsilon) = 26$, giving
$\Delta_2(K) \approx -13.1628$, the best local 
minimum and gives $C_2(K) = 1.951261$. 
For the ramified case $D_K=4 \cdot 20220595$, we obtained
$\delta_2(\varepsilon) = 28$, $C_2(K) = 2.284033$.

\smallskip
But the case $D_K=81624 = 8 \cdot 3 \cdot 19 \cdot 179$, for which $h_K=8$,
with the valuation $v_p(\order {\mathcal T}_K) = 20$, gives $C_2(K) = 2.4514$ 
and shows, once again, that genera theory may modify the results for $p=2$
and more generally for $p \mid d$. Note that in 
the above results, there is no solution $D_K \in [20406, 10^6]$.
To illustrate this, we use the same program for $D_K \in [81628, \, 5 \cdot 10^5]$:

\footnotesize
\begin{verbatim}
D=81628     m=20407    vptor=2     Ymin=6.15838824...   Cp=0.2451
D=81640     m=20410    vptor=4     Ymin=4.15849428...   Cp=0.4902
D=81713     m=81713    vptor=5     Ymin=3.15913899...   Cp=0.6128
D=81788     m=20447    vptor=7     Ymin=1.15980078...   Cp=0.8578
D=82684     m=20671    vptor=8     Ymin=0.16766028...   Cp=0.9794
D=83144     m=20786    vptor=9    Ymin=-0.82833773...   Cp=1.1013
D=84361     m=84361    vptor=10   Ymin=-1.81785571...   Cp=1.2221
D=86284     m=21571    vptor=11   Ymin=-2.80159728...   Cp=1.3417
D=100045    m=100045   vptor=14   Ymin=-5.69485522...   Cp=1.6857
D=115005    m=115005   vptor=16   Ymin=-7.59433146...   Cp=1.9034
D=376264    m=94066    vptor=17   Ymin=-7.73930713...   Cp=1.8357
D=495957    m=495957   vptor=19   Ymin=-9.54007224...   Cp=2.0084
D=1476668   m=369167   vptor=20   Ymin=-9.75304296...   Cp=1.9518
\end{verbatim}
\normalsize

\subsubsection{\sc Program for $p \in [3, 50]$}\label{v3tor}
In this case, genera theory does not intervenne. 
We do not write the cases where $v_p(\order {\mathcal T}_K)=0$ ($p$-rational fields).
The constant $C_p(K)$ has some variations for very small $D_K$ 
but stabilizes and seems locally decreasing for larger $D_K$; 
so we mention the maximal ones, but the last value
is more significant to evaluate an upperbound:

\footnotesize
\begin{verbatim}
{n=16;bD=5;BD=10^6;forprime(p=3,50,print(" ");print("p=",p);ymin=10;
for(D=bD,BD,e=valuation(D,2);M=D/2^e;if(core(M)!=M,next);
if((e==1 || e>3)||(e==0 & Mod(M,4)!=1)||(e==2 & Mod(M,4)==1),next);
m=D;if(e!=0,m=D/4);P=x^2-m;K=bnfinit(P,1);Kpn=bnrinit(K,p^n);
C5=component(Kpn,5);Hpn0=component(C5,1);Hpn=component(C5,2);
Hpn1=component(Hpn,1);vptor=valuation(Hpn0/Hpn1,p);
Y=log(sqrt(D))/log(p)-vptor;if(Y<ymin,ymin=Y;Cp=vptor*log(p)/log(sqrt(D));
print("D=",D," m=",m," vptor=",vptor," Ymin=",Y," Cp=",Cp))))}
      p=3
D=24       m=6         vptor=1      Ymin=0.44639463...   Cp=0.6913
D=29       m=29        vptor=2     Ymin=-0.46747762...   Cp=1.3050
D=105      m=105       vptor=3     Ymin=-0.88189136...   Cp=1.4163
D=488      m=122       vptor=4     Ymin=-1.18266604...   Cp=1.4197
D=1213     m=1213      vptor=6     Ymin=-2.76826302...   Cp=1.8565
D=1896     m=474       vptor=7     Ymin=-3.56498395...   Cp=2.0378
D=13861    m=13861     vptor=8     Ymin=-3.65959960...   Cp=1.8431
D=21713    m=21713     vptor=10    Ymin=-5.45532735...   Cp=2.2003
D=168009   m=168009    vptor=11    Ymin=-5.52410420...   Cp=2.0088
D=321253   m=321253    vptor=12    Ymin=-6.22909046...   Cp=2.0793
      p=5
D=53       m=53        vptor=1      Ymin=0.23344053...   Cp=0.8107
D=73       m=73        vptor=2     Ymin=-0.66709383...   Cp=1.5005
D=217      m=217       vptor=3     Ymin=-1.32864091...   Cp=1.7949
D=1641     m=1641      vptor=4     Ymin=-1.70010976...   Cp=1.7392
D=25037    m=25037     vptor=5     Ymin=-1.85352571...   Cp=1.5890
D=71308    m=17827     vptor=6     Ymin=-2.52836443...   Cp=1.7283
D=304069   m=304069    vptor=7     Ymin=-3.07782014...   Cp=1.7847
(...)
D=4788645  m=4788645   vptor=10    Ymin=-5.22138818...   Cp=2.0926 
      p=7
D=24       m=6         vptor=1     Ymin=-0.18340170...   Cp=1.2246
D=145      m=145       vptor=2     Ymin=-0.72123238...   Cp=1.5640
D=797      m=797       vptor=3     Ymin=-1.28335992...   Cp=1.7476
D=30556    m=7639      vptor=4     Ymin=-1.34640462...   Cp=1.5074
D=92440    m=23110     vptor=5     Ymin=-2.06196222...   Cp=1.7018
D=287516   m=71879     vptor=6     Ymin=-2.77039718...   Cp=1.8578
(...)
D=4354697  m=4354697   vptor=7     Ymin=-3.07207825...   Cp=1.7821
      p=11
D=29       m=29        vptor=1     Ymin=-0.29786428...   Cp=1.4242
D=145      m=145       vptor=2     Ymin=-0.96227041...   Cp=1.9272
D=424      m=106       vptor=3     Ymin=-1.73853259...   Cp=2.3781
D=35068    m=8767      vptor=4     Ymin=-1.81786877...   Cp=1.8330
D=163873   m=163873    vptor=5     Ymin=-2.49637793...   Cp=1.9971
      p=13
D=8        m=2         vptor=1     Ymin=-0.59464276...   Cp=2.4669
D=2285     m=2285      vptor=3     Ymin=-1.49234424...   Cp=1.9898
D=98797    m=98797     vptor=4     Ymin=-1.75808000...   Cp=1.7842
D=382161   m=382161    vptor=5     Ymin=-2.49437601...   Cp=1.9955
      p=17
D=69       m=69        vptor=2     Ymin=-1.25277309...   Cp=2.6765
D=3209     m=3209      vptor=3     Ymin=-1.57516648...   Cp=2.1055
D=8972     m=2243      vptor=4     Ymin=-2.39372069...   Cp=2.4902
D=1631753  m=1631753   vptor=5     Ymin=-2.47545212...   Cp=1.9805
      p=19
D=109      m=109       vptor=1     Ymin=-0.20335454...   Cp=1.2552
D=193      m=193       vptor=2     Ymin=-1.10633396...   Cp=2.2379
D=2701     m=2701      vptor=3     Ymin=-1.65825418...   Cp=2.2359
(...)
D=1482837  m=1482837   vptor=4     Ymin=-1.58706704...   Cp=1.6577
D=6839105  m=6839105   vptor=5     Ymin=-2.32747604...   Cp=1.8709
D=8736541  m=8736541   vptor=5     Ymin=-2.28589639...   Cp=1.8422
      p=23
D=140      m=35        vptor=1     Ymin=-0.21198348...   Cp=1.2690
D=493      m=493       vptor=2     Ymin=-1.01123893...   Cp=2.0227
D=10433    m=10433     vptor=3     Ymin=-1.52451822...   Cp=2.0332
D=740801   m=740801    vptor=4     Ymin=-1.84475964...   Cp=1.8559
      p=29
D=33       m=33        vptor=1     Ymin=-0.48081372...   Cp=1.9261
D=41       m=41        vptor=2     Ymin=-1.44858244...   Cp=3.6270
D=53093    m=53093     vptor=4     Ymin=-2.38448997...   Cp=2.4759
D=30596053 m=30596053  vptor=5     Ymin=-2.44061964...   Cp=1.9536
      p=31
D=8        m=2         vptor=1     Ymin=-0.69722637...   Cp=3.3028
D=6168     m=1542      vptor=2     Ymin=-0.72930075...   Cp=1.5739
D=90273    m=90273     vptor=3     Ymin=-1.33857946...   Cp=1.8056
D=1294072  m=323518    vptor=4     Ymin=-1.95087990...   Cp=1.9520
      p=37
D=33       m=33        vptor=1     Ymin=-0.51584228...   Cp=2.0654
D=3340     m=835       vptor=2     Ymin=-0.87650089...   Cp=1.7801
D=124129   m=124129    vptor=3     Ymin=-1.37588711...   Cp=1.8471
      p=41
D=73       m=73        vptor=1     Ymin=-0.42232716...   Cp=1.7311
D=2141     m=2141      vptor=2     Ymin=-0.96743241...   Cp=1.9369
D=187113   m=187113    vptor=3     Ymin=-1.36552680...   Cp=1.8354
      p=43
D=88       m=22        vptor=1     Ymin=-0.40479944...   Cp=1.6801
D=6520     m=1630      vptor=2     Ymin=-0.83246977...   Cp=1.7130
D=283596   m=70899     vptor=3     Ymin=-1.33094416...   Cp=1.7974
      p=47
D=301      m=301       vptor=1     Ymin=-0.25884526...   Cp=1.3492
D=26321    m=26321     vptor=2     Ymin=-0.67821659...   Cp=1.5131
D=368013   m=368013    vptor=3     Ymin=-1.33566464...   Cp=1.8025
\end{verbatim}
\normalsize

\noindent
The interval $[2, 10^6]$ was not always sufficient (see the cases 
$p=5, 7, 19, 29$, above).
For instance for $p=7$, we ignore if the bound $C_p(K)=1.8578$ can be 
exceeded; we have computed up to $D_K \leq 2 \cdot 10^7$, where 
$v_p(\order {\mathcal T}_K)$ takes at most the values $6$ or $7$ 
with $C_p(K) < 1.7821$. 
So $v_p(\order {\mathcal T}_K) \geq 8$ does exist for greater discriminants, but 
$\frac{8 \cdot {\rm log}_\infty(7)}{{\rm log}_\infty(\sqrt{2 \cdot 10^7})} 
\approx 1.8520$, which is significant of the evolution of $C_p(K)$
as $D_K\to \infty$.

\smallskip \noindent
The same program with $p=3$, $n >18$, taking discriminants, 
$D_K \in [10^6, 2.5 \cdot10^7]$ then in $[10^8, 5 \cdot10^6]$
(two days of computer for each part), gives ($p=3$):
\footnotesize
\begin{verbatim}   
D=1000005    m=1000005    vptor=1      Ymin=5.28771209...    Cp=0.1590
D=1000049    m=1000049    vptor=2      Ymin=4.28773212...    Cp=0.3180
D=1000104    m=250026     vptor=3      Ymin=3.28775715...    Cp=0.4771
D=1000133    m=1000133    vptor=4      Ymin=2.28777034...    Cp=0.6361
D=1000169    m=1000169    vptor=5      Ymin=1.28778673...    Cp=0.7951
D=1000380    m=250095     vptor=6      Ymin=0.28788273...    Cp=0.9542
D=1001177    m=1001177    vptor=8     Ymin=-1.71175481...    Cp=1.2722
D=1014693    m=1014693    vptor=9     Ymin=-2.70565175...    Cp=1.4298
D=1074760    m=268690     vptor=10    Ymin=-3.67947724...    Cp=1.5821
D=1185256    m=296314     vptor=11    Ymin=-4.63493860...    Cp=1.7281
D=2354577    m=2354577    vptor=12    Ymin=-5.32254344...    Cp=1.7970
D=6099477    m=6099477    vptor=13    Ymin=-5.88934151...    Cp=1.8282
D=13495160   m=3373790    vptor=14    Ymin=-6.52791825...    Cp=1.8736
D=21242636   m=5310659    vptor=16    Ymin=-8.32143995...    Cp=2.0837
(...)
D=100025621  m=100025621  vptor=13    Ymin=-4.61627031...    Cp=1.5506
D=104326449  m=104326449  vptor=16    Ymin=-7.59711043...    Cp=1.9041
\end{verbatim}
\normalsize

\smallskip\noindent
The case $D_K=21242636$ leads  to $C_3(K)=2.0837$; but it is difficult to 
predict the behavior of $C_3$ at infinity. In the second part, no data between 
the two discriminants, which suggests an irregular decreasing of $C_3(K)$
as $D_K\to\infty$.

\begin{remark}
From these calculations in the quadratic case, one may consider,
in an heuristic framework, that we have the good following lower 
bounds for ${\mathcal C}_p$:

\smallskip \noindent
${\mathcal C}'_3 \approx 2.0837$, ${\mathcal C}'_5 \approx 2.0926$, 
${\mathcal C}'_7 \approx 1.8578$, 
${\mathcal C}'_{11} \approx 1.9971$, ${\mathcal C}'_{13} \approx 1.9955$, 
${\mathcal C}'_{17} \approx 1.9805$, 
${\mathcal C}'_{19} \approx 2.2379$, ${\mathcal C}'_{23} \approx 1.8559$, 
${\mathcal C}'_{29} \approx 2.4759$,
${\mathcal C}'_{31} \approx 1.9520$, ${\mathcal C}'_{37} \approx 1.8471$, 
${\mathcal C}'_{41} \approx 1.8354$, 
${\mathcal C}'_{43} \approx 1.7974$, ${\mathcal C}'_{47} \approx 1.8025$.
\end{remark}

\subsubsection{\sc Remarks and Heuristics}\label{vph}
Let ${\mathcal K}_{\rm real}^{(2)}$ be the family of real quadratic fields; we consider 
$C_p(K)$ and try to understand its behavior regarding $p$ and $D_K$:

\smallskip
(i) For $p \gg 0$, an estimation of ${\mathcal C}_p^{(2)} := 
\sup_{K \in {\mathcal K}_{\rm real}^{(2)}}(C_p(K))$
is more difficult and, a fortiori, for $\limsup_{K \in {\mathcal K}_{\rm real}^{(2)}}(C_p(K))$; 
for instance, we have found that for $\Q(\sqrt{19})$ and $p_0=13599893$, one has
$v_{p_0}(\order {\mathcal T}_{\Q(\sqrt{19})})=1$, whence 
${\mathcal C}_{p_0}^{(2)} \geq 7.5855$.
The following program can be used for {\it huge values} of $p$ to find quadratic 
fields $K$ such that $v_p(\order {\mathcal R}_K) \geq 1$; in practice one never finds 
$v_p(\order {\mathcal R}_K) \geq 2$ for ``usual'' discriminants. However, for these 
solutions, one must compute $v_p(\order {\mathcal T}_K)$ with the classical program
of Section \ref{section3} to be sure of the result (we treat separately the case $p_0 \mid D_K$).

\smallskip
\footnotesize
\begin{verbatim}
{p=13599893;pp=p^2;for(D=5,5*10^8,e=valuation(D,2);M=D/2^e;if(core(M)!=M,next);
if((e==1||e>3)||(e==0 & Mod(M,4)!=1)||(e==2 & Mod(M,4)==1),next);s=kronecker(D,p);
if(s==0,next);E=quadunit(D);nu=norm(E);u=(1+nu-nu*s+s)/2;P=component(E,1)+Mod(0,pp);
e1=component(E,2);e2=component(E,3);A=Mod(e1+e2*x,P)^(p-s)-u;if(A==0,print(D)))}
\end{verbatim}
\normalsize

\smallskip\noindent
The {\it next discriminants} $D_K > 4\cdot 19$, up to $5 \cdot 10^8$ 
(more that two days of computer), for which 
$v_{p_0}(\order {\mathcal T}_K) \geq 1$ (in fact $=1$), are:

\smallskip\noindent
$37473505$, $45304189$, $104143053$,  $111800589$, $112985161$,
$181148197$,  $239100989$,  
 
\hfill $288517452$, $350532569$, $387058008$, $414929433$, $477524401$,

\smallskip\noindent
giving $C_{p_0}(K)$ =
$1.8837$, $1.8635$, $1.7794$, $1.7726$, $1.7716$, $1.7276$, 
$1.7028$,  $1.6864$, $1.6697$, $1.6613$, $1.6550$, $1.6438$, respectively.

\smallskip\noindent
Thus we notice, as expected, a significant decrease of the function 
$C_{p_0}(K)$ since we did not find any $v_{p_0}(\order {\mathcal T}_K) > 1$, 
until $D_K \leq 5 \cdot 10^8$, knowing that other quadratic fields with 
{\it arbitrary} $v_{p_0}(\order {\mathcal T}_K)$ exist with huge discriminants, as:
\ctl{$D_K=p_0^4+4 = 34209124997537575597791879605$, for which 
$C_{p_0}(K)= 0.4999$.}
This field is the first element of families
$K=\Q \Big (\sqrt{a^2 \!\cdot\! p_0^{2\rho} + b^2}\Big)$, $a \geq 1$, $b \in \{1, 2\}$,
described in Subsection \ref{family}, for which 
$\delta_{p_0}(\varepsilon_K)=\rho-1$, whence
$v_{p_0}(\order {\mathcal T}_K) \geq \rho-1$ and $C_p(K)< 1+ o(1)$.
Note that for $\rho-1=10$ and $p_0=13599893$, $D_K \approx 10^{157}$.

\smallskip\noindent
Unfortunately, we ignore what happens for $5 \cdot 10^8 < D_K < p_0^4+4$ 
because of the order of magnitude; to get $C_{p_0}(K)<1.3$, we must have 
for instance $v_{p_0}(\order {\mathcal T}_K)=1$ and
$D_K > 94334377272$, then $D_K > 9333929793774$ to get $C_{p_0}(K)<1.1$.

\smallskip\noindent
We then have the following alternative: either $C_{p_0}(K) < 7.5855$ for 
all $D_K>4 \cdot 19$, whence ${\mathcal C}_{p_0}^{(2)}=7.5855$, 
or ${\mathcal C}_{p_0}^{(2)}$ is greater than $7.5855$ or infinite.

\smallskip \noindent
The existence of infinitely many $K \in {\mathcal K}_{\rm real}^{(2)}$ such that 
$C_{p_0}(K) > 7.5855$ remains possible but assumes the strong condition
$v_{p_0}(\order {\mathcal T}_K) \!> 0.4618 \cdot {\rm log}_\infty(\sqrt{D_K})$ 
for infinitely many $K \in {\mathcal K}_{\rm real}^{(2)}$. 

\smallskip\noindent
The most credible case should be that, for each $p$, {\it there exist finitely 
many} $K \in {\mathcal K}_{\rm real}^{(2)}$ for which $v_{p}(\order {\mathcal T}_K)
\gg {\rm log}_\infty(\sqrt{D_K})$, whence $C_p(K) \gg 
{\rm log}_\infty(p)$; so for ``almost all'' $K \in {\mathcal K}_{\rm real}^{(2)}$, 
we would have $C_p(K) \ll 1$ (and often $0$ as explained in (iii)), except for 
some critical infinite families for which $C_p(K) \leq 1+o(1)$; if there is no other 
possibilities, ${\mathcal C}_p^{(2)}$ does exist and is equal 
to $\max_{D_K \leq D_0}^{} (C_p(K))$ for a sufficiently large $D_0$.

\medskip
(ii) The existence of ${\mathcal C}_p$ (over ${\mathcal K}_{\rm real}$)
essentially depends on $v_p(\order {\mathcal R}_K)$
since the influence of $v_p(\order \Cl_K^{\rm c})$ seems negligible, which is 
reinforced by classical heuristics on class groups \cite{CL, CM}, or by specific results
in suitable towers \cite [Proposition 7.1]{TV}, then, mainely, by strong conjectures
(and partial proofs) in \cite{EV} as 
$\order \Cl_K \ll_{\epsilon, p, d}\big( \sqrt{\vert D_K \vert} \big)^\epsilon$ 
for any number field of degree $d$, i.e., for all $\epsilon>0$ the 
existence of $C_{\epsilon, p, d}$ such that:
\ctl{${\rm log}_\infty(\order \Cl_K) \leq 
{\rm log}_\infty(C_{\epsilon, p, d})+\epsilon  \cdot {\rm log}_\infty(\sqrt{\vert D_K \vert})$,}
strengthening the classical Brauer theorem (existence of 
an universal constant ${\mathcal C}_0$ such that, 
${\rm log}_\infty(h_K) \leq {\mathcal C}_0 \cdot {\rm log}_\infty(\sqrt{\vert D_K\vert})$
for all number field~$K$); for quadratic and cyclic cubic fields, ${\mathcal C}_0=1$ 
(Remark \ref{note1}).

\medskip
(iii) For any fixed $p$, $\ds\liminf_{K \in {\mathcal K}_{\rm real}^{(2)}} (C_p(K) )=0$ 
(see Byeon \cite[Theorem 1.1]{B}, after Ono, where a lower boud of the 
density of $p$-rational fields is given for $p>3$). Indeed,
as $D_K \to \infty$, statistically, ``almost all'' real quadratic fields $K$ are 
such that $\order {\mathcal T}_K = 1$.

\medskip
(iv) Now, if $K$ is fixed and $p\to\infty$, $\ds\liminf_{p} (C_p(K) )=0$.
One may see this as an unproved generalization, for $v_p(\order {\mathcal R}_K)$,
of theorems of Silverman \cite{Si}, Graves--Murty \cite{GM} and others
about Fermat quotients of rationals, showing 
the considerable difficulties of such subjects, despite the numerical obviousness
since in practice, ``for almost all $p$'', $v_p(\order {\mathcal T}_K)=0$.
We have conjectured, after numerous calculations and heuristics, 
that, for $K \in {\mathcal K}_{\rm real}$ 
fixed, the set of primes $p$, such that ${\mathcal T}_K \ne 1$,
is finite \cite[Conjecture 8.11]{Gr2}, i.e., $C_p(K)=0$ for all $p \gg 0$;
otherwise $\ds\limsup_{p} (C_p(K) )=\infty$.
If this conjecture is false  for the field $K$, there exists 
an infinite set of prime numbers $p_i$ such that $v_{p_i}(\order {\mathcal T}_{K}) \geq 1$ 
giving $C_{p_i}(K) \geq \Frac{{\rm log}_\infty(p_i)}{{\rm log}_\infty(\sqrt{D_{K}})}$ 
arbitrary large as $i\to \infty$. But this is not incompatible with the existence, for 
each $i$, of ${\mathcal C}_{p_i} < \infty$; indeed, in that case, 
$C_{p_i}(K)$ may be very large with decreasing values of the
$C_{p_i}(K')$, for $D_{K'} \gg D_{K}$ as shown, for instance in 
${\mathcal K}_{\rm real}^{(2)}$, by the example given in (i). 

\smallskip\noindent
If, on the contrary, the conjecture is true over 
${\mathcal K}_{\rm real}^{(2)}$ (or more generaly over ${\mathcal K}_{\rm real}$), 
for each fixed non-$p$-rational field $K$,
let $p_K^{}=\ds\sup_{{\mathcal T}_{K,p} \ne 1} (p)$; then it will be interesting 
to have a great lot of $C_{p_K^{}}(K)$, which is of course non-effective.

\subsection{A special family of quadratic fields}\label{family}

Consider, for $p$ fixed, the field:
\ctl{$K=\Q(\sqrt{a^2 \!\cdot\! p^{2\rho} + 1})$, with $\rho \geq 2$, 
$a\geq 1$, $p \nmid a$ ;} 
assuming that $m:=a^2 \cdot p^{2\rho} + 1$ is 
a squarefree integer, its fondamental unit is $\varepsilon_K = a \cdot p^\rho + \sqrt m$ 
and $D_K=m$ (for $a\! \cdot\! p$ even) or $4\,m$  (for $a \cdot p$ odd); 
the case of $m=a^2 \!\cdot\! p^{2\rho} + 4$ would be similar. 
From the formula \eqref{lou}, we have
$h_K < \Frac{1}{2} \cdot \sqrt{D_K}$, 
and an upper bound being $a \cdot p^\rho$, this allows to get
$v_p(\order \Cl_K) \leq \rho + \frac{{\rm log}_\infty(a)}{{\rm log}_\infty(p)}$ to take into 
account the possible (incredible) case where $h_K$ is a maximal $p$th~power. 
As $\delta_p(\varepsilon_K)+ v_2(\order {\mathcal W}_K)=\rho-1$ 
for these fields, it follows:
\ctl{$\rho-1 \leq  v_p(\order {\mathcal T}_K) = 
v_p(\order \Cl_K) + \delta_p(\varepsilon_K) + v_p(\order {\mathcal W}_K) <  
2\,\rho + \frac{{\rm log}_\infty(a)}{{\rm log}_\infty(p)}$.}
Thus, since $\Frac{{\rm log}_\infty(\sqrt {D_K})}{{\rm log}_\infty(p)} 
\approx  \rho + \frac{{\rm log}_\infty(a)}{{\rm log}_\infty(p)}$, 
we have proved, in this particular case, that:

\centerline{$\Frac{\rho-1}{\rho + \frac{{\rm log}_\infty(2 \, a)}{{\rm log}_\infty(p)}} 
\leq C_p(K) < \Frac{2\,\rho + \frac{{\rm log}_\infty(a)}{{\rm log}_\infty(p)}}
{\rho + \frac{{\rm log}_\infty(a)}{{\rm log}_\infty(p)}} \in [1, 2[$.}

\medskip
We shall assume the conjecture that, for all $p$, 
$m :=a^2 \cdot p^{2\rho} + 1$ is squarefree\,\footnote{\,The conjecture
is true for integers of the form $n^2+1$ \cite{H-B}, but we ignore if this remains true
for $n = a \cdot p^{\rho}$, $p$ prime, $\rho \in \N$, $a \geq 1$; 
but this is not so essential (see Remark \ref{nonqf}).} for infinitely many integers 
$\rho \geq 2$. Whence the partial result:

\begin{theorem} 
Let ${\mathcal K}_{\rm real}^{(2)}$ be the family of real quadratic fields and let:
$$C_p(K) := \Frac{v_p(\order {\mathcal T}_K) \cdot 
{\rm log}_\infty(p)}{{\rm log}_\infty(\sqrt {D_K})},\ \ 
\hbox{for $K \in {\mathcal K}_{\rm real}^{(2)}$ and $p \geq 2$.}$$
Then, under the above conjecture on $m :=a^2 \cdot p^{2\rho} + 1$, $\rho \geq 2$,
one has, for each fixed $p$, $C_p(K)\in [0, 2[$ for an infinite subset
of ${\mathcal K}_{\rm real}^{(2)}$.
\end{theorem}

Moreover, if we consider the estimation of $v_p(\order \Cl_K)$
largely excessive, as explained in the \S\,\ref{vph}\,(ii), one may 
conjecture that, for the above family of fields
$K=\Q(\sqrt{a^2 \cdot p^{2\rho} + 1})$,  $\rho \geq 2$, one has:
\ctl{$\rho-1 \leq v_p(\order {\mathcal T}_K) < \rho\cdot (1 +o(1))$,}
and the statement of the theorem becomes:

\medskip\noindent
{\it For each $p \geq 2$, $C_p(K)$ is asymptotically equal to $1$ for 
an infinite subset of~${\mathcal K}_{\rm real}^{(2)}$.}

\medskip \noindent
Indeed, $v_p(\order {\mathcal T}_K)$ (in $vptor$)
and $v_p(\order \Cl_K)$ (in $vph$) are given by the following program, to 
illustrate the relation $\rho-1 \leq v_p(\order {\mathcal T}_K) < \rho\cdot (1 +o(1))$. 

\smallskip \noindent
We vary $p$ and $\rho$ in intervals such that, for instance, ${\rm log}_\infty(m)<40$ 
(just choose $a$, $n$ large enough, and copy and paste the program to get  
complete tables):

\smallskip
\footnotesize
\begin{verbatim}
{a=1;B=40;n=26;forprime(p=2,20,for(rho=2,B/(2*log(p)),m=a^2*p^(2*rho)+1;
if(core(m)!=m,next);D=m;if(Mod(m,4)!=1,D=4*m);P=x^2-m;K=bnfinit(P,1);
Kpn=bnrinit(K,p^n);C5=component(Kpn,5);Hpn0=component(C5,1);
Hpn=component(C5,2);Hpn1=component(Hpn,1);vptor=valuation(Hpn0/Hpn1,p);
Cp=vptor*log(p)/log(sqrt(D));h=component(component(component(K,8),1),1);
vph=valuation(h,p);
print("p=",p," m=",m," rho=",rho," vptor=",vptor," Cp=",Cp," vph=",vph)))}

      a=1, p=2, D=m
m=17    rho=2    vptor=1 Cp=0.4893010842... vph=0
m=65    rho=3    vptor=3 Cp=0.9962858772... vph=1
(...)
m=4398046511105    rho=21    vptor=29    Cp=1.3809523809...    vph=10
m=17592186044417    rho=22    vptor=24    Cp=1.0909090909...    vph=3
(...)
m=18014398509481985    rho=27    vptor=29    Cp=1.074074074...    vph=6
m=72057594037927937     rho=28    vptor=26    Cp=0.9285714285...   vph=2

      a=1, p=3, D=4*m
m=82    rho=2    vptor=1    Cp=0.3792886959...    vph=0
m=730    rho=3    vptor=3    Cp=0.8260927150...    vph=1
(...)
m=16677181699666570    rho=17    vptor=17    Cp=0.9642146068...    vph=1
m=150094635296999122    rho=18    vptor=19    Cp=1.0198095452...    vph=2

      a=1, p=5, D=4*m
m=626    rho=2    vptor=1    Cp=0.4113240423...    vph=0
m=15626    rho=3    vptor=2    Cp=0.5829720101...    vph=0
(...)
m=2384185791015626    rho=11    vptor=11    Cp=0.9623227412...    vph=1
m=59604644775390626    rho=12    vptor=11    Cp=0.8849075871...    vph=0

      a=2, p=3, D=m
m=2917    rho=3    vptor=3    Cp=0.8261991487...    vph=1
m=26245    rho=4    vptor=3    Cp=0.6478156494...    vph=0
(...)
m=66708726798666277    rho=17    vptor=16    Cp=0.9074961005...    vph=0
m=600378541187996485    rho=18    vptor=19    Cp=1.0198095452...    vph=2

      a=2, p=5, D=m
m=2501    rho=2    vptor=1    Cp=0.4113870622...    vph=0
m=62501    rho=3    vptor=2    Cp=0.5829745440...    vph=0
(...)
m=9536743164062501    rho=11    vptor=10    Cp=0.8748388557...    vph=0
m=238418579101562501    rho=12    vptor=11    Cp=0.8849075871...    vph=0
\end{verbatim}
\normalsize

For $K=\Q(\sqrt{a^2 \cdot p^{2\rho} + 4})$, $a$ odd, 
$\varepsilon_K = \Frac{a \cdot p^\rho + \sqrt m}{2}$,
$K$ is unramified at $2$ giving a maximal $C_p(K)=1.2222222215...$
(for $a=1$, $p=3$, $\rho=9$, $vptor=11$, $vph=3$):

\footnotesize
\begin{verbatim}
      a=1, p=3, D=m
m=85 rho=2 vptor=1 Cp=0.4945750747656077295917504 vph=0
m=733 rho=3 vptor=3 Cp=0.9991705549452351082457751 vph=1
(...)
m=109418989131512359213 rho=21 vptor=23 Cp=1.0952380952380952380943703 vph=3
m=984770902183611232885 rho=22 vptor=22 Cp=0.9999999999999999999999159 vph=1

      a=1, p=5, D=m
m=629 rho=2 vptor=1 Cp=0.4995050064384236683280022 vph=0
m=15629 rho=3 vptor=2 Cp=0.6666489958698626477868625 vph=0
(...)
m=37252902984619140629 rho=14 vptor=13 Cp=0.9285714285714285714263589 vph=0
m=931322574615478515629 rho=15 vptor=16 Cp=1.0666666666666666666665717 vph=2

      a=1, p=7, D=m
m=2405 rho=2 vptor=1 Cp=0.4998930943437939009946102 vph=0
m=117653 rho=3 vptor=2 Cp=0.6666647253436162691864834 vph=0
(...)
m=3909821048582988053 rho=11 vptor=10 Cp=0.9090909090909090908873656 vph=0
m=191581231380566414405 rho=12 vptor=12 Cp=0.9999999999999999999995529 vph=1

      a=1, p=11, D=m
m=14645 rho=2 vptor=1 Cp=0.4999857604139424915125214 vph=0
m=1771565 rho=3 vptor=2 Cp=0.6666665620428398909421335 vph=0
(...)
m=5559917313492231485 rho=9 vptor=10 Cp=1.1111111111111111110925908 vph=2
m=672749994932560009205 rho=10 vptor=9 Cp=0.8999999999999999999998884 vph=0

      a=1, p=17, D=m
m=24137573 rho=3 vptor=2 Cp=0.6666666601676951315812133 vph=0
m=6975757445 rho=4 vptor=3 Cp=0.7499999999810259247791427 vph=0
(...)
m=168377826559400933 rho=7 vptor=6 Cp=0.8571428571428571423437840 vph=0
m=48661191875666868485 rho=8 vptor=8 Cp=0.9999999999999999999981866 vph=1
\end{verbatim}
\normalsize

\smallskip
One sees, from these excerpts, the weak influence of $vph= v_p(\Cl_K)$ 
giving very few $C_p(K)=1+o(1)$. Larger values of $a$, $p$, yields the 
same kind of results. 

\begin{remark} \label{nonqf}
Without assuming that $m =a^2 \cdot p^{2\rho} \pm 1$ 
(or $m =a^2 \cdot p^{2\rho} \pm 4$) is squarefree (which is indeed impossible for minus signs), 
the same program gives always $C_p(K)$ near $1$ and in any case in $[0, 2[$ 
as far as we have tested this property; of course, if $m= b^2\,m'$ with $m'$ squarefree, the unit 
$\varepsilon' = a \cdot p^\rho + b \cdot \sqrt {m'}$ is not necessarily 
fundamental so that $\delta_p(\varepsilon_K) \leq \delta_p(\varepsilon')$ 
and $D_K = m'$ or $4\,m'$ may be very small (the program deals only 
with non-squarefree integers $m$):

\smallskip
\footnotesize
\begin{verbatim}
{B=60;for(a=1,18,forprime(p=2,19,for(rho=1,B/(2*log(p)),m=a^2*p^(2*rho)+1;
n=rho+6;if(core(m)!=m,P=x^2-m;K=bnfinit(P,1);D=component(component(K,7),3);
Kpn=bnrinit(K,p^n);C5=component(Kpn,5);Hpn0=component(C5,1);
Hpn=component(C5,2);Hpn1=component(Hpn,1);vptor=valuation(Hpn0/Hpn1,p);
Cp=vptor*log(p)/log(sqrt(D));
print("a=",a," p=",p," m=",m," rho=",rho," vptor=",vptor," Cp=",Cp)))))}
\end{verbatim}
\normalsize

\smallskip\noindent
Then the biggest $C_p(K)$ are for trivial cases 
($m=5^2\cdot 41$ and $m=250001=53^2 \cdot 89$):

\footnotesize
\begin{verbatim}
a=1   p=2    D=m=1025      rho=5    vptor=4     Cp=1.4932
a=4   p=5    D=m=250001    rho=3    vptor=2     Cp=1.4342
\end{verbatim}
\end{remark}

\subsection{Reciprocal study}\label{reciprocal}
We fix $p \geq 2$, $\rho \geq 2$, and we try to build units 
of the form $\eta = 1 + p^{\rho} \cdot (X + Y \cdot \sqrt m)$,
where $X, Y \in \Z$ and where $m$ is a squarefree integer.
It is not necessary to consider the case $\frac{X + Y \cdot \sqrt m}{2}$,
$X$ and $Y$ of same parity for $m \equiv 1 \pmod 4$,
since this only concerns the cases $p=2$ (in which case this can
modify $\rho$ into $\rho-1$) and $p=3$ (since any cube
of unit is of the suitable form and this also modifies the choice of $\rho$).

\smallskip
In $K=\Q(\sqrt m)$, $\eta$ may be 
a $p$-power of the fundamental unit $\varepsilon_K$, but this 
goes in the good direction to get an upper bound of $C_p(K)$,
if we use $\delta_p(\eta)$ instead of $\delta_p(\varepsilon_K)$
to compute $v_p(\order {\mathcal T}_K)$,
since $\delta_p(\varepsilon_K) \leq \delta_p(\eta)$.
\begin{lemma}
The number $\eta = 1 + p^{\rho} \cdot (X + Y \cdot \sqrt m)$,
$X, Y \in \Z$, is a unit of $\Q(\sqrt m)$ if and only if $X= p^\rho \cdot a$ and
$a \cdot (2 + p^{2\rho} \cdot a) = m \cdot b^2$
(resp. $a \cdot (1 + 2^{2\rho-2} \cdot a) = m \cdot b^2$)
if $p \ne 2$ (resp. $p=2$), $a, b \in \Z$.
\end{lemma}

\begin{proof}
We have ${\rm N}_{K/\Q}(\eta)=\pm 1$ if and only if:
\ctl{$1+  p^{\rho} \cdot (X + Y \cdot \sqrt m) + p^{\rho} \cdot (X - Y \cdot \sqrt m)
+ p^{2 \rho} \cdot (X^2 - m \cdot Y^2)= \pm 1$} 
which is equivalent (since $-1$ is absurd for $\rho \geq 2$) to 
$2\cdot X + p^\rho \cdot X^2 = m\cdot p^\rho \cdot Y^2$.
For $p\ne 2$, this yields $X=  p^\rho \cdot a$, $Y=b$, such that
$a\cdot (2 + p^{2 \rho} \cdot a) = m \cdot b^2$. For $p=2$, one must consider
the relation $a\cdot (1 + 2^{2 \rho - 2} \cdot a) = m \cdot b^2$, whence
in practice the relation $a\cdot (1 + 2^{2 \rho} \cdot a) = m \cdot b^2$
replacing $\rho$ by $\rho-1$.
\end{proof}

So, we shall fix $\rho$ large enough, increase $a$ 
in some interval and write $a \cdot (2 + p^{2 \rho} \cdot a)$ 
(resp. $a \cdot (1 + 2^{2 \rho} \cdot a)$) under the form 
$m \cdot b^2$, $m$ sqarefree. We then compute the successive minima 
of $D_K$ for $K=\Q(\sqrt m)$, to try to get maximal values for $C_p(K)$:

\footnotesize
\begin{verbatim}
{p=3;rho=21;n=rho+6;ba=10^8+1;Ba=2*10^8;pp=p^(2*rho);Dmin=10^100;d=2;
if(p==2,d=1);for(a=ba,Ba,B=a*(d+pp*a);m=core(B);D=m;if(Mod(m,4)!=1,D=4*m);
if(D<Dmin,Dmin=D;b=component(core(B,1),2);P=x^2-m;K=bnfinit(P,1);
Kpn=bnrinit(K,p^n);C5=component(Kpn,5);Hpn0=component(C5,1);Hpn=component(C5,2);
Hpn1=component(Hpn,1);vptor=valuation(Hpn0/Hpn1,p);Cp=vptor*log(p)/log(sqrt(D));
h=component(component(component(K,8),1),1);vph=valuation(h,p);
print("D=",D," a=",a," b=",b," vptor=",vptor," vph=",vph," Cp=",Cp)))}
\end{verbatim}
\normalsize

\noindent
We have done a great lot of experimentations with very large discriminants 
without obtainig any $C_p(K) >2$, except, for $p=2$ and the known case
(see \S\,\ref{v2tor}):

\footnotesize
\begin{verbatim}
D=81624    a=9728    b=557872    vptor=20    vph=3    Cp=2.45147522
\end{verbatim}
\normalsize

\noindent
which corresponds to a too small discriminant since the stabilisation of
$C_p(K)$ seems better and better as soon as $D_K \gg 0$. Moreover,
$v_2(\Cl_K)=3$ in this example.

\smallskip
Let $a \in [10^8+1, 2 \cdot 10^8]$ (an interval of negative values 
of $a$ gives similar results):

\footnotesize
\begin{verbatim}
      p=3, rho=21
D                                     a          b      vptor   vph    Cp
4376759652795686111245843894049436844 100000001  1         22   2      0.5729
1094189935082719682370900209849436840 100000002  2         21   0      0.5560
6474496916274063005939132968034008    100000004  26        21   1      0.5926
(...)
780348725011642441673212              100250343  2374203   21   0      0.8387
97192908950160977396761               100966886  3387724   21   1      0.8717
\end{verbatim}
\normalsize

\smallskip \noindent
There is no solution $a \in  [10^8+966886 ,2 \cdot 10^8]$ giving 
smaller discriminants.

\smallskip
\footnotesize
\begin{verbatim}
      p=2, rho=30, n=2*rho
D                                     a          b     vptor   vph     Cp
11529215276652771834290899906846977   100000001  1        35   5       0.6186
17055053207700727651215465398745      100000004  26       42   11      0.8096
(...)
48025975228418415280613               100175668  490822   37   6       0.9821
28578131029527067857561               100311617  637139   34   4       0.9115
617974038061148975453                 100469200  4339580  36   4       1.0424
\end{verbatim}
\normalsize

Same remarks as for the case $p=3$; despite genera theory, it seems that $C_p(K)$
remains close to 1 and is not increasing substantially in the process.

\section{Numerical investigations for cyclic cubic fields}

For the computations in the set ${\mathcal K}_{\rm ab}^{(3)}$ of
cyclic cubic fields, we shall use the
direct calculation of $\order {\mathcal T}_K$ from the program 
testing the $p$-rationality, taking $n$ large~enough. 

\smallskip
See \cite{HZ} for statistics on $v_p(R_{K,p}) = v_p(\order {\mathcal R}_K)+2$ 
(resp. $v_p(\order {\mathcal R}_K)+1$) in the non-ramified (resp. ramified) case
for cyclic cubic fields of conductors up to $10^8$; this gives, for cubic fields, 
the analogue of the computation of $\delta_p(\varepsilon)$ for quadratic fields 
in Subsection \ref{max}. 

\smallskip
Note that, due to Galois action, the integers $v_p(\order {\mathcal T}_K)$ 
are even if $p \equiv 2 \pmod 3$ and arbitrary if not (same remark for 
$v_p(\order \Cl_K)$ and $v_p(\order {\mathcal R}_K)$); then
$v_2(\order {\mathcal W}_K)=2$ if $2$ splits in $K$, otherwise
$v_2(\order {\mathcal W}_K)=0$ and $v_p(\order {\mathcal W}_K)=0$
for $p>2$.

\subsection{Maximal values of $v_p(\order {\mathcal T}_K)$}

The program uses the well-known classification of cyclic cubic 
fields \cite{ET} with conductor $f_K \leq B\!f$ (see the formulas 
\ref{polcubic} giving the corresponding polynomials 
defining $K$), and processes as for the quadratic case.
We give first the case $p = 3$ to see the influence of genera theory;
we compute the successive maxima of $v_p(\order {\mathcal T}_K)$ 
(in $vptor$) with the corresponding $f_K$ and the polynomial 
defining the field of conductor $f_K$. We print in the first line the 
maximal value obtained for $C_p(K)$ in the selected interval.

\smallskip
Recall that $D_K = f_K^2$, where
$f_K=f'_K$ or $9 \cdot f'_K$ with $f'_K=\ell_1 \cdots \ell_t$, for
distinct primes $\ell_i \equiv 1 \pmod 3$:

\smallskip
\footnotesize
\begin{verbatim}
{p=3;n=26;bf=7;Bf=10^7;Max=0;for(f=bf,Bf,e=valuation(f,3);if(e!=0 & e!=2,next);
F=f/3^e;if(Mod(F,3)!=1 || core(F)!=F,next);F=factor(F);Div=component(F,1);
d=component(matsize(F),1);for(j=1,d-1,D=component(Div,j);if(Mod(D,3)!=1,break));
for(b=1,sqrt(4*f/27),if(e==2 & Mod(b,3)==0,next);A=4*f-27*b^2;
if(issquare(A,&a)==1,if(e==0,if(Mod(a,3)==1,a=-a);
P=x^3+x^2+(1-f)/3*x+(f*(a-3)+1)/27);
if(e==2,if(Mod(a,9)==3,a=-a);P=x^3-f/3*x-f*a/27);
K=bnfinit(P,1);Kpn=bnrinit(K,p^n);C5=component(Kpn,5);Hpn0=component(C5,1);
Hpn=component(C5,2);Hpn1=component(Hpn,1);
vptor=valuation(Hpn0/Hpn1,p);Cp=vptor*log(p)/log(f);
if(vptor>Max,Max=vptor;print("f=",f," vptor=",vptor," P=",P," Cp=",Cp)))))}
      p=3          Cp=1.1492
f=19        vptor=1    P=x^3 + x^2 - 6*x - 7
f=199       vptor=2    P=x^3 + x^2 - 66*x + 59
f=427       vptor=4    P=x^3 + x^2 - 142*x - 680
f=1843      vptor=5    P=x^3 + x^2 - 614*x + 3413
f=2653      vptor=6    P=x^3 + x^2 - 884*x - 8352
f=17353     vptor=7    P=x^3 + x^2 - 5784*x - 145251
f=30121     vptor=8    P=x^3 + x^2 - 10040*x + 306788
f=114079    vptor=9    P=x^3 + x^2 - 38026*x + 2822399
f=126369    vptor=10   P=x^3 - 42123*x + 3046897
f=355849    vptor=11   P=x^3 + x^2 - 118616*x - 15235609
f=371917    vptor=12   P=x^3 + x^2 - 123972*x + 15854684
f=1687987   vptor=15   P=x^3 + x^2 - 562662*x - 116533621   
     p=2, n=36    Cp=1.2475
f=31        vptor=2    P=x^3 + x^2 - 10*x - 8
f=171       vptor=6    P=x^3 - 57*x - 152
f=2689      vptor=8    P=x^3 + x^2 - 896*x + 5876
f=6013      vptor=12   P=x^3 + x^2 - 2004*x - 32292
f=6913      vptor=13   P=x^3 + x^2 - 2304*x - 256
f=311023    vptor=16   P=x^3 + x^2 - 103674*x + 5068523
f=544453    vptor=18   P=x^3 + x^2 - 181484*x - 19862452
f=618093    vptor=24   P=x^3 - 206031*x + 21289870          
      p=7         Cp=1.3955
f=9         vptor=1    P=x^3 - 3*x + 1
f=313       vptor=2    P=x^3 + x^2 - 104*x + 371
f=721       vptor=3    P=x^3 + x^2 - 240*x - 988
f=1381      vptor=4    P=x^3 + x^2 - 460*x - 1739
f=29467     vptor=6    P=x^3 + x^2 - 9822*x - 20736
f=177541    vptor=7    P=x^3 + x^2 - 59180*x + 3051075
f=1136587   vptor=10   P=x^3 + x^2 - 378862*x + 58428991    
\end{verbatim}
\normalsize

\subsection{Experiments for a conjectural upper bound -- Cubic fields}
In the same way as for quadratic fields, we give, for each prime $p$, 
the successive minima of 
$\Delta_p(K) = \Frac{{\rm log}_\infty(f_K)}{{\rm log}_\infty(p)} 
- v_p(\order {\mathcal T}_K)$ (in $Y\!min$) with the value of
$C_p(K)=\Frac{v_p(\order {\mathcal T}_K) \cdot {\rm log}_\infty(p)}
{{\rm log}_\infty(f_K)}$ (in $Cp$), obtained for some polynomial $P$ and the 
corresponding conductor $f_K$:

\smallskip
\footnotesize
\begin{verbatim}
{n=36;bf=7;Bf=5*10^6;forprime(p=2,50,ymin=10;print("p="p);for(f=bf,Bf,
e=valuation(f,3);if(e!=0 & e!=2,next);F=f/3^e;if(Mod(F,3)!=1||core(F)!=F,next);
F=factor(F);Div=component(F,1);d=component(matsize(F),1);
for(j=1,d-1,D=component(Div,j);if(Mod(D,3)!=1,break));
for(b=1,sqrt(4*f/27),if(e==2 & Mod(b,3)==0,next);A=4*f-27*b^2;
if(issquare(A,&a)==1,if(e==0,if(Mod(a,3)==1,a=-a);
P=x^3+x^2+(1-f)/3*x+(f*(a-3)+1)/27);
if(e==2,if(Mod(a,9)==3,a=-a);P=x^3-f/3*x-f*a/27);
K=bnfinit(P,1);Kpn=bnrinit(K,p^n);C5=component(Kpn,5);Hpn0=component(C5,1);
Hpn=component(C5,2);Hpn1=component(Hpn,1);vptor=valuation(Hpn0/Hpn1,p);
Y=log(f)/log(p)-vptor;if(Y<ymin,ymin=Y;print(P);Cp=vptor*log(p)/log(f);
print("f=",f," vptor=",vptor," Ymin=",Y," Cp=",Cp))))))}
\end{verbatim}
\normalsize

\noindent
The first minimum occurs for $f := f_K=7$ and $vptor := v_p(\order {\mathcal T}_K)=0$; 
we omit these cases of $p$-rationality. For some $p$, we have been obliged to 
consider larger conductors $f$ to get significant solutions, especially for
$p=11$ for which the first non-trivial example is for $f=5000059$
and $P=x^3 + x^2 - 1666686\,x - 408523339$.

\footnotesize
\begin{verbatim}
    p=2,  Cp=1.247565
P=x^3 - 57*x - 152
f=171        vptor=6         Ymin=1.41785251...    Cp=0.8088
P=x^3 + x^2 - 2004*x - 32292
f=6013       vptor=12        Ymin=0.55386924...    Cp=0.9559
P=x^3 + x^2 - 2304*x - 256
f=6913       vptor=14       Ymin=-1.24490378...    Cp=1.0976
P=x^3 - 206031*x + 21289870
f=618093     vptor=24       Ymin=-4.76253559...    Cp=1.2475
    p=3,  Cp=1.149252
P=x^3 + x^2 - 6*x - 7
f=19         vptor=1         Ymin=1.68014385...    Cp=0.3731
P=x^3 + x^2 - 142*x - 680
f=427        vptor=4         Ymin=1.51312239...    Cp=0.7255
P=x^3 + x^2 - 884*x - 8352
f=2653       vptor=6         Ymin=1.17582211...    Cp=0.8361
P=x^3 - 42123*x + 3046897
f=126369     vptor=10        Ymin=0.69254513...    Cp=0.9352
P=x^3 + x^2 - 118616*x - 15235609
f=355849     vptor=11        Ymin=0.63491606...    Cp=0.9454
P=x^3 + x^2 - 123972*x + 15854684
f=371917     vptor=12       Ymin=-0.32488392...    Cp=1.0278
P=x^3 + x^2 - 562662*x - 116533621
f=1687987    vptor=15       Ymin=-1.94803671...    Cp=1.1492
    p=5,  Cp=1.462906
P=x^3 + x^2 - 50*x - 123
f=151        vptor=2         Ymin=1.11741123...    Cp=0.6415
P=x^3 + x^2 - 1002*x + 6905
f=3007       vptor=4         Ymin=0.97608396...    Cp=0.8038
P=x^3 + x^2 - 2214*x + 19683
f=6643       vptor=8        Ymin=-2.53143306...    Cp=1.4629
    p=7,  Cp=1.395563
P=x^3 - 3*x + 1
f=9          vptor=1         Ymin=0.12915006...    Cp=0.8856
P=x^3 + x^2 - 460*x - 1739
f=1381       vptor=4        Ymin=-0.28422558...    Cp=1.0765
P=x^3 + x^2 - 9822*x - 20736
f=29467      vptor=6        Ymin=-0.71145865...    Cp=1.1345
P=x^3 + x^2 - 59180*x + 3051075
f=177541     vptor=7        Ymin=-0.78853291...    Cp=1.1269
P=x^3 + x^2 - 378862*x + 58428991
f=1136587    vptor=10       Ymin=-2.83443766...    Cp=1.3955
    p=11,  Cp=0.621490
P=x^3 + x^2 - 1666686*x - 408523339
f=5000059    vptor=2         Ymin=4.43270806...    Cp=0.3109
P=x^3 - 1680483*x - 503584739
f=5041449    vptor=4         Ymin=2.43614601...    Cp=0.6215
    p=13,  Cp=1.632521
P=x^3 + x^2 - 20*x - 9
f=61         vptor=1         Ymin=0.60271151...    Cp=0.6239
P=x^3 + x^2 - 196*x - 349
f=589        vptor=2         Ymin=0.48676495...    Cp=0.8042
P=x^3 + x^2 - 1064*x + 12299
f=3193       vptor=3         Ymin=0.14576042...    Cp=0.9536
P=x^3 + x^2 - 1824*x + 8919
f=5473       vptor=4        Ymin=-0.64415121...    Cp=1.1919
P=x^3 + x^2 - 19920*x + 615317
f=59761      vptor=7        Ymin=-2.71215372...    Cp=1.6325
    p=17,  Cp=0.910481
P=x^3 - 399*x - 3059
f=1197       vptor=2         Ymin=0.50160254...    Cp=0.7994
P=x^3 - 84837*x + 1046323
f=254511     vptor=4         Ymin=0.39327993...    Cp=0.9105
    p=19,  Cp=0.974463
P=x^3 + x^2 - 30*x + 27
f=91         vptor=1         Ymin=0.53199286...    Cp=0.6527
P=x^3 + x^2 - 404*x + 629
f=1213       vptor=2         Ymin=0.41161455...    Cp=0.8293
P=x^3 - 3477*x - 26657
f=10431      vptor=3         Ymin=0.14237703...    Cp=0.9547
P=x^3 + x^2 - 1213944*x - 503921781
f=3641833    vptor=5         Ymin=0.13102760...    Cp=0.9744
    p=23,  Cp=0.880087
P=x^3 + x^2 - 1060*x - 11428
f=3181       vptor=2         Ymin=0.57214663...    Cp=0.7775
P=x^3 + x^2 - 515154*x - 19633104
f=1545463    vptor=4         Ymin=0.54500411...    Cp=0.8801
    p=29,  Cp=1.569666
P=x^3 + x^2 - 24*x - 27
f=73         vptor=2        Ymin=-0.72584422...    Cp=1.5696
    p=31,  Cp=0.981745
P=x^3 + x^2 - 30*x + 27
f=91         vptor=1         Ymin=0.31359240...    Cp=0.7613
P=x^3 - 12027*x + 388873
f=36081      vptor=3         Ymin=0.05578357...    Cp=0.9817
    p=37,  Cp=1.119764 
P=x^3 - 39*x - 26
f=117        vptor=1         Ymin=0.31882641...    Cp=0.7582
P=x^3 + x^2 - 5300*x + 119552
f=15901      vptor=3        Ymin=-0.32086480...    Cp=1.1197
    p=41,  Cp=0.976052
P=x^3 + x^2 - 672*x - 2764
f=2017       vptor=2         Ymin=0.04906930...    Cp=0.9760
    p=43,  Cp=0.914939
P=x^3 + x^2 - 20*x - 9
f=61         vptor=1         Ymin=0.09296866...    Cp=0.9149
    p=47,  Cp=0.878952
P=x^3 + x^2 - 2126*x + 11813
f=6379       vptor=2         Ymin=0.27543656...    Cp=0.8789
\end{verbatim}
\normalsize

\section{Examples of non-Galois totally real number fields}
We shall consider (non necessarily Galois) cubic fields, 
with an approach using randomness.
The tested polynomials of dgree $3$ define almost always
Galois groups isomorphic to $S_3$. It is more difficult to find non-$p$-rational 
fields for large $p$ and to obtain a lower bound of ${\mathcal C}_p^{(3)}$ for the 
family ${\mathcal K}_{\rm real}^{(3)}$ of totally real cubic fields.

\subsection {Program for a given cubic polynomial and increasing $p$}
The program concerns fields $K$ defined by $P=x^3+a\,x^2+b\,x+1$,
for random $a, b$ and increasing $p$ in $[2, 10^5]$.
It tests the irreducibility of $P$ and that $D_K>0$ (real roots). 
We give only the non-$p$-rational cases for which one prints 
the corresponding $C_p(K)$. 

\footnotesize
\begin{verbatim}
{n=4;N=100;bp=2;Bp=10^5;ymin=10;a=random(N);b=random(N);P=x^3+a*x^2+b*x+1;
if(polisirreducible(P)==1 & poldisc(P)>0,print(P);K=bnfinit(P,1);
D=component(component(K,7),3);forprime(p=bp,Bp,Kpn=bnrinit(K,p^n);
C5=component(Kpn,5);Hpn0=component(C5,1);Hpn=component(C5,2);
Hpn1=component(Hpn,1);vptor=valuation(Hpn0/Hpn1,p);Y=log(sqrt(D))/log(p)-vptor;
if(vptor > 0 & Y<ymin,ymin=Y;Cp=vptor*log(p)/log(sqrt(D));
print("p=",p," vptor=",vptor," Ymin=",Y," Cp=",Cp))))}
\end{verbatim}
\normalsize

\noindent
We obtain, after several tries and $p$ up to $10^5$,
omitting the small values of $C_p(K)$:

\smallskip
\footnotesize
\begin{verbatim}
    P=x^3 + 21*x^2 + 47*x + 1
p=11          vptor=1         Ymin=1.75210757...
p=523         vptor=1         Ymin=0.05426629...
p=3517        vptor=1        Ymin=-0.19179768...    Cp=1.2373
p=173483      vptor=1        Ymin=-0.45297114...    Cp=1.8280
    P=x^3 + 19*x^2 + 51*x + 1
p=487         vptor=1        Ymin=-0.40614414...    Cp=1.6839
    P=x^3 + 92*x^2 + 52*x + 1
p=18637       vptor=1        Ymin=-0.14697706...    Cp=1.1723
    P=x^3 + 99*x^2 + 23*x + 1
p=73          vptor=1         Ymin=0.47867182... 
p=15803       vptor=1        Ymin=-0.34379282...    Cp=1.5239
p=145259      vptor=1       Ymin =-0.46625984...    Cp=1.8735
p=622519      vptor=1       Ymin =-0.52447869...    Cp=2.1029
    P=x^3 + 98*x^2 + 62*x + 1
p=3           vptor=2         Ymin=3.86940839...
p=61          vptor=1         Ymin=0.56857262...
p=37549       vptor=1        Ymin=-0.38783270...    Cp=1.63354
    P=x^3 + 87*x^2 + 74*x + 1
p=5441        vptor=1         Ymin=0.01344518...    Cp=0.9867
    P=x^3 + 73*x^2 + 67*x + 1
p=6133        vptor=1        Ymin=-0.03273397...    Cp=1.0338
    P=x^3 + 19*x^2 + 83*x + 1
p=61          vptor=1        Ymin=-0.52664318...    Cp=2.1126
p=5419        vptor=1        Ymin=-0.77366996...    Cp=4.4183
p=12703       vptor=1        Ymin=-0.79407472...    Cp=4.8561
\end{verbatim}
\normalsize

\smallskip\noindent
Note that by accident, $P=x^3 + 19\,x^2 + 83\,x + 1$, with a large 
$C_{12703}(K) \approx 4.8561$, defines the cyclic cubic field $K$ 
of conductor $7$ (in some sense, an analogue of $K=\Q(\sqrt {19})$
with $p_0=13599893$ for which $C_{p_0}(K) \approx 7.5856$, 
see \S\,\ref{vph}\,(i)). 

\smallskip \noindent
But the forthcoming conductors $f>7$, up to $4 \cdot10^6$, 
give decreasing $C_{12703}(K)$, as shown by the following excerpts, 
where no $v_p(\order {\mathcal T}_K) \geq 2$ were found with $p=12703$:

\smallskip
\footnotesize
\begin{verbatim}
f=7         vptor=1   x^3 + x^2 - 2*x - 1                Cp=4.856130
f=17767     vptor=1   x^3 + x^2 - 5922*x + 17109         Cp=0.965712
f=54649     vptor=1   x^3 + x^2 - 18216*x - 931057       Cp=0.866244
f=101839    vptor=1   x^3 + x^2 - 33946*x + 1059880      Cp=0.819484
(...)
f=497647    vptor=1   x^3 + x^2 - 165882*x + 7114509     Cp=0.720372
f=547903    vptor=1   x^3 + x^2 - 182634*x - 12804696    Cp=0.715127
(...)
f=859621    vptor=1   x^3 + x^2 - 286540*x + 49348613    Cp=0.691556
f=865189    vptor=1   x^3 + x^2 - 288396*x - 7818745     Cp=0.691229
(...)
f=1680543   vptor=1   x^3 - 560181*x + 55084465          Cp=0.659214
f=1744477   vptor=1   x^3 + x^2 - 581492*x - 143305555   Cp=0.657501
(...)
f=2477313   vptor=1   x^3 - 825771*x + 262870435         Cp=0.641839
f=2486871   vptor=1   x^3 - 828957*x - 138988457         Cp=0.641671
(...)
f=3616141   vptor=1   x^3 + x^2 - 1205380*x + 483625376  Cp=0.625762
f=3628081   vptor=1   x^3 + x^2 - 1209360*x - 96883200   Cp=0.625626
(...)
f=4036591   vptor=1   x^3 + x^2 - 1345530*x + 122293757  Cp=0.621237
f=4037779   vptor=1   x^3 + x^2 - 1345926*x - 499488217  Cp=0.621225
\end{verbatim}
\normalsize

\subsection{Program for a given $p$ and random cubic polynomials}

The program tries polynomials in a random way, so that
the discriminants are not obtained in the natural order; we then write, in the first line,
the largest $C_p(K)$ obtained:

\footnotesize
\begin{verbatim}
{p=3;N=1000;n=18;ymin=10;for(k=1,10^6,a=random(N);b=random(N);c=random(N);
P=x^3+a*x^2+b*x+c;if(polisirreducible(P)==1 & poldisc(P)>0,K=bnfinit(P,1);
D=component(component(K,7),3);Kpn=bnrinit(K,p^n);C5=component(Kpn,5);
Hpn0=component(C5,1);Hpn=component(C5,2);Hpn1=component(Hpn,1);
vptor=valuation(Hpn0/Hpn1,p);Y=log(sqrt(D))/log(p)-vptor;
if(vptor>0 & Y<ymin,ymin=Y;Cp=vptor*log(p)/log(sqrt(D));
print("P=",P," vptor=",vptor," Ymin=",Y," Cp=",Cp))))}
    p=2      Cp=1.497370
P=x^3 + 315*x^2 + 151*x + 13      vptor=6      Ymin=4.62049695...
P=x^3 + 44*x^2 + 388*x + 962      vptor=7      Ymin=2.65795067...
P=x^3 + 78*x^2 + 498*x + 584      vptor=6      Ymin=2.33817139...
P=x^3 + 473*x^2 + 759*x + 90      vptor=12     Ymin=1.79924824...
P=x^3 + 176*x^2 + 760*x + 472     vptor=14    Ymin=-0.65040380...
P=x^3 + 30*x^2 + 165*x + 220      vptor=12    Ymin=-3.98594984...
    p=3      Cp=1.042763
P=x^3 + 57*x^2 + 251*x + 70       vptor=4      Ymin=2.95145981...
P=x^3 + 93*x^2 + 396*x + 396      vptor=4      Ymin=2.08419811...
P=x^3 + 53*x^2 + 602*x + 140      vptor=6      Ymin=1.91171871...
P=x^3 + 143*x^2 + 672*x + 617     vptor=8      Ymin=1.71414906...
P=x^3 + 360*x^2 + 698*x + 132     vptor=4      Ymin=1.11320078...
P=x^3 + 194*x^2 + 649*x + 440     vptor=7      Ymin=1.02340828...
P=x^3 + 38*x^2 + 343*x + 722      vptor=6      Ymin=0.41712275...
P=x^3 + 77*x^2 + 512*x + 874      vptor=8     Ymin=-0.32807458...
    p=5      Cp=1.238605
P=x^3 + 177*x^2 + 590*x + 456     vptor=1      Ymin=4.94615149...
P=x^3 + 222*x^2 + 789*x + 180     vptor=2      Ymin=1.62797441...
P=x^3 + 45*x^2 + 362*x + 772      vptor=3      Ymin=1.32811388...
P=x^3 + 83*x^2 + 400*x + 251      vptor=2      Ymin=1.22069007...
P=x^3 + 197*x^2 + 718*x + 508     vptor=8     Ymin=-1.54112474...
    p=7      Cp=1.201178
P=x^3 + 784*x^2 + 964*x + 288     vptor=1      Ymin=3.97483926...
P=x^3 + 505*x^2 + 710*x + 134     vptor=2      Ymin=2.57552488...
P=x^3 + 73*x^2 + 492*x + 196      vptor=3      Ymin=1.85163167...
P=x^3 + 57*x^2 + 695*x + 263      vptor=1      Ymin=1.35093638...
P=x^3 + 95*x^2 + 839*x + 252      vptor=5      Ymin=0.64570147...
P=x^3 + 114*x^2 + 804*x + 142     vptor=2     Ymin=-0.37221306...
P=x^3 + 97*x^2 + 829*x + 122      vptor=5     Ymin=-0.83742084...
    p=19     Cp=1.139412
P=x^3 + 50*x^2 + 631*x + 470      vptor=1      Ymin=1.58556226...
P=x^3 + 57*x^2 + 777*x + 801      vptor=1      Ymin=1.54028119...
P=x^3 + 549*x^2 + 732*x + 39      vptor=3      Ymin=0.69038895...
P=x^3 + 93*x^2 + 891*x + 383      vptor=2      Ymin=0.64611301...
P=x^3 + 123*x^2 + 375*x + 217     vptor=1      Ymin=0.46422353...
P=x^3 + 226*x^2 + 777*x + 408     vptor=2      Ymin=0.20875475...
P=x^3 + 196*x^2 + 849*x + 918     vptor=2     Ymin=-0.24470848...
    p=1009    Cp=1.227512
P=x^3 + 171*x^2 + 667*x + 604     vptor=1      Ymin=0.49598190...
P=x^3 + 89*x^2 + 567*x + 36       vptor=1      Ymin=0.37961552...
P=x^3 + 54*x^2 + 435*x + 719      vptor=1      Ymin=0.29433117...
P=x^3 + 93*x^2 + 636*x + 944      vptor=1      Ymin=0.07490160...
P=x^3 + 432*x^2 + 347*x + 19      vptor=1     Ymin=-0.06432442...
P=x^3 + 130*x^2 + 942*x + 899     vptor=1     Ymin=-0.06692434...
P=x^3 + 70*x^2 + 553*x + 735      vptor=1     Ymin=-0.18534377...
\end{verbatim}
\normalsize

\begin{remarks}
(i) The case $p=2$ with $P=x^3 + 30\,x^2 + 165\,x + 220$, where:
\ctlm{$v_2(\order {\mathcal T}_K)=12  \ \  \& \ \  \Delta_2(K) \approx -3.98595$,}
seems exceptional, but the discriminant $D_K=66825$ is rather small. The
Galois closure $L$ of $K$ contains $\Q(\sqrt {33})$ and is defined by the polynomial:
\ctlm{$Q=x^6 - 60 x^5 +1131 x^4 - 6380 x^3 - 15708 x^2 + 145200 x + 170368$;}
then $v_2(\order {\mathcal T}_L)=25$, giving $C_2(L) \approx 1.3476$
instead of $C_2(K) \approx 1.4973$.

\smallskip
(ii) For $p=5$ and $P=x^3 + 197\,x^2 + 718\,x + 508$, 
$v_5(\order {\mathcal T}_K)=8$ is large, but
$D_K=1069350637=769 \cdot 1390573$ is rather large, giving
$C_5(K) \approx 1.2386$.

\smallskip
(iii) For $p=7$, $P=x^3 + 95\,x^2 + 839\,x + 252$,  $v_7(\order {\mathcal T}_K)=5$,
with $C_7(K) \approx 0.8856$,
but $D_K=3486121421$, while for $P=x^3 + 114\,x^2 + 804\,x + 142$,
$v_7(\order {\mathcal T}_K)=2$ with $C_7(K) \approx 1.2286$, but $D_K=564$.

\smallskip
(iv) We have computed $C_p(L)$ for the Galois closure $L$ of
the above fields $K$ (Galois group $S_3$). The values $C_p(L)$ are smaller, although the
$v_p(\order {\mathcal T}_L)$ are roughly speaking twice of $v_p(\order {\mathcal T}_K)$
(cf. Example (i)). This reinforces the idea that extensions $L/K$ may give in general
values of $C_p(L)$ smaller than those of $C_p(K)$.
\end{remarks}

\section{Conjectures on $v_p(\order {\mathcal T}_K)$}

\subsection{$p$-adic statements}
The numerical results (quadratic and cubic cases, with the particular family
of quadratic fields studied in Subsections \ref{numexp}, \ref{family}, \ref{reciprocal}) 
suggest the following conjecture that we state in its strongest form; 
we shall discuss about some conditions of application of such a conjecture, 
for instance assuming that the fields $K$ are of given degree or are elements of
specified families.

\smallskip
The points (i) and (ii) are equivalent statements:

\begin{conjecture} \label{conjprinc}
Let $K \in {\mathcal K}_{\rm real}$ (or any element of a specified family 
${\mathcal K} \subseteq {\mathcal K}_{\rm real}$), and let $p \geq 2$ be 
a prime number. Let ${\mathcal T}_K$ be the torsion group of the Galois 
group of  the maximal abelian $p$-ramified pro-$p$-extension of $K$ 
(under Leopoldt's conjecture).

\smallskip
(i) There exists a constant ${\mathcal C}_p({\mathcal K}) =: {\mathcal C}_p$,
independent of $K \in {\mathcal K}$, such that:
\begin{equation*}\label{pconjecture}
v_p(\order {\mathcal T}_K) \leq {\mathcal C}_p \cdot 
\Frac{{\rm log}_\infty(\sqrt{D_K})}{{\rm log}_\infty(p)}, \ 
\hbox{for all $K \in {\mathcal K}$}.
\end{equation*}
\quad (ii) The residue $\wt \kappa_{K,p}^{}$ of the normalized
$\zeta$-function $\wt \zeta_{K,p}(s) = \Frac{p \cdot [K \cap \,\Q^{\rm c} : \Q]}
{2^{d-1}} \,\zeta_{K, p}(s)$ at $s=1$ (see Subsection \ref{places}), is 
conjecturaly such that:

\smallskip
\centerline{$v_p(\wt \kappa_{K,p}^{})  \leq 
{\mathcal C}_p \cdot\Frac{{\rm log}_\infty(\sqrt{D_K})}{{\rm log}_\infty(p)}$,
for all $K \in {\mathcal K}$.}
\end{conjecture}

We may propose the following conjecture which takes into account
the numerical behaviour of the $C_p(K)$ that we have observed;
but unfortunately, this would need inaccessible computations to 
be more convincing:

\begin{conjecture}\label{conjprinc2}
Let ${\mathcal K}_{\rm real}$ be the set of all totally real number 
fields and let $p \geq 2$ be any fixed prime number. Then
$\ds \limsup_{K \in {\mathcal K}_{\rm real}, D_K \to \infty} 
\Big(\Frac {v_p(\order {\mathcal T}_K) \cdot {\rm log}_\infty(p)}
{{\rm log}_\infty(\sqrt{D_K})} \Big) = 1$.
\end{conjecture}

\begin{theorem} Let $d$ be a fixed positive integer and let $p \nmid d$.
Let ${\mathcal K}_{\rm ab}^{(d)}$ be the set of real abelian extension of $\Q$
whose degree divides $d$. Then the conjecture \ref{conjprinc} is true for 
${\mathcal K}_{\rm ab}^{(d)}$ if and only if it is true for the subset of 
cyclic extensions of ${\mathcal K}_{\rm ab}^{(d)}$.
\end{theorem}

\begin{proof} 
Let $K \in {\mathcal K}_{\rm ab}^{(d)}$. As $p \nmid [K : \Q]$, 
${\mathcal T}_K \simeq \bigoplus_\chi {\mathcal T}_K^{e_\chi}$,
where $\chi$ runs trough the set of irreducible rational characters of 
${\rm Gal}(K/\Q)$ (a set which is in bijection with that of cyclic subfields of $K$), 
$e_\chi$ being the corresponding idempotent; then ${\mathcal T}_K^{e_\chi}$
is isomorphic to a submodule of ${\mathcal T}_{k_\chi}$,
where $k_\chi$ (cyclic) is the subfield of $K$ fixed by the kernel of $\chi$,
and $v_p(\order {\mathcal T}_K) = \sum_\chi v_p(\order {\mathcal T}_K^{e_\chi})$.
We have: 
$$C_p(K) = \Frac{v_p(\order {\mathcal T}_K) \cdot {\rm log}_\infty(p)}
{{\rm log}_\infty(\sqrt{D_K})} = \ds \sum_\chi 
\Frac{v_p(\order {\mathcal T}_K^{e_\chi}) \cdot {\rm log}_\infty(p)}
{{\rm log}_\infty(\sqrt{D_K})} \leq \ds \sum_\chi 
\Frac{v_p(\order {\mathcal T}_{k_\chi}) \cdot {\rm log}_\infty(p)}
{{\rm log}_\infty(\sqrt{D_K})}; $$
but $D_K=D_{k_\chi}^{[K : k_\chi]} \cdot {\rm N}_{k_\chi/\Q} (D_{K/k_\chi})$
yields ${\rm log}_\infty(\sqrt{D_K}) \geq[K : k_\chi] \cdot{\rm log}_\infty \big (\sqrt{D_{k_\chi}} \big )$
for all $\chi$. Thus, if we have the inequalities 
$C_p(k_\chi) =\Frac{v_p(\order {\mathcal T}_{k_\chi}) \cdot {\rm log}_\infty(p)}
{{\rm log}_\infty \big (\sqrt{D_{k_\chi}} \big )} \leq {\mathcal C}_p$ for all~$\chi$, the theorem 
follows with a constant $ {\mathcal C}'_p$, depending on the maximal number of cyclic 
subfields for elements of the set ${\mathcal K}_{\rm ab}^{(d)}$, which may be explicited.
\end{proof}

Let's illustrate this by means of random real biquadratic fields $K$ for which 
we compute the invariants of $K$ and its subfields (then $vptor = v1+v2+v3$ for $p\ne 2$):

\footnotesize
\begin{verbatim}
{p=3;n=18;N=2*10^2;B=10^6;vmax=0;for(j=1,B,m1=random(N)+1;m2=random(N)+1;
P1=x^2-m1;P2=x^2-m2;P3=x^2-m1*m2;P=component(polcompositum(P1,P2),1);
if(poldegree(P)!=4,next);D1=nfdisc(P1);D2=nfdisc(P2);D3=nfdisc(P3);D=nfdisc(P);
K1=bnfinit(P1,1);Kpn=bnrinit(K1,p^n);C5=component(Kpn,5);Hpn0=component(C5,1); 
Hpn=component(C5,2);Hpn1=component(Hpn,1);v1=valuation(Hpn0/Hpn1,p); 
K2=bnfinit(P2,1);Kpn=bnrinit(K2,p^n);C5=component(Kpn,5);Hpn0=component(C5,1); 
Hpn=component(C5,2);Hpn1=component(Hpn,1);v2=valuation(Hpn0/Hpn1,p); 
K3=bnfinit(P3,1);Kpn=bnrinit(K3,p^n);C5=component(Kpn,5);Hpn0=component(C5,1); 
Hpn=component(C5,2);Hpn1=component(Hpn,1);v3=valuation(Hpn0/Hpn1,p); 
K=bnfinit(P,1);Kpn=bnrinit(K,p^n);C5=component(Kpn,5);Hpn0=component(C5,1); 
Hpn=component(C5,2);Hpn1=component(Hpn,1);vptor=valuation(Hpn0/Hpn1,p); 
Cp1=v1*log(p)/log(sqrt(D1));Cp2=v2*log(p)/log(sqrt(D2));
Cp3=v3*log(p)/log(sqrt(D3));Cp=vptor*log(p)/log(sqrt(D));
if(vptor>vmax,vmax=vptor;print(D1," ",D2," ",D3," ",D," ",
v1," ",v2," ",v3," ",vptor," ",Cp1," ",Cp2," ",Cp3," ",Cp)))}

D1    D2    D3      D              v1 v2 v3  vptor  Cp1     Cp2     Cp3      Cp
41    840   34440   1186113600     0  0  2   2      0       0       0.4206   0.2103 
12    1896  632     14379264       0  7  0   7      0       2.0378  0        0.9332
1896  1096  32469   67471101504    7  0  1   8      2.0378  0       0.2115   0.7049
1896  13    24648   607523904      7  0  2   9      2.0378  0       0.4345   0.9777
1976  1896  234156  877264517376   2  7  1   10     0.5790  2.0378  0.1777   0.7989
1896  824   97644   152549611776   7  4  0   11     2.0378  1.3090  0        0.9385
1896  488   14457   13376310336    7  4  1   12     2.0378  1.4197  0.2293   1.1308 
449   1896  851304  724718500416   1  7  5   13     0.3597  2.0378  0.8045   1.0459
\end{verbatim}
\normalsize

\noindent
For two random discriminants of quadratic fields, taken up to $2\cdot 10^2$, 
the program did not find any $v_3(\order {\mathcal T}_K) > 13$. We have 
$C_p(K)<\max(C_p(K_1),C_p(K_2),C_p(K_3))$ (obvious for the biquadratic case).
It is likely that the compositum $K$ of two fields $K_1$, $K_2$, gives in general 
{\it smaller} $C_p(K)$, except if $v_p(\order {\mathcal T}_{K_1})$ and 
$v_p(\order {\mathcal T}_{K_2})$ are small regarding $v_p(\order {\mathcal T}_{K})$ 
and if the number of subfields of $K$ is important, but in that case $C_p(K)$
remains very small, as is shown by the following rare examples obtained
as compositum of two random non-Galois cubic fields giving large 
$v_p(\order {\mathcal T}_{K})$ (the last line gives $v1$, $v2$,
$vptor$, $Cp1$, $Cp2$, $Cp$):

\footnotesize
\begin{verbatim}
     p=2
P1=x^3-45*x^2+24*x-1,P2=x^3-36*x^2+27*x-1,P=x^9+27*x^8-2844*x^7-54486*x^6
   +2141829*x^5+20969253*x^4-10466577*x^3-5546475*x^2+1542807*x+10233
766017   77433  23187342173591131003005670474209    
1      1      9       0.102317   0.123147      0.172756
P1=x^3-12*x^2+9*x-1,P2=x^3-20*x^2+23*x-1,P=x^9-24*x^8-192*x^7+5728*x^6
   +10131*x^5-301710*x^4+238483*x^3+148968*x^2-83460*x- 8520
3753   15465  21724158202972986227625    
1      1      10      0.168437   0.143712      0.269535
P1=x^3-23*x^2+22*x-1,P2=x^3-19*x^2+42*x-1,P=x^9+12*x^8-634*x^7-4844*x^6
   +112245*x^5+317540*x^4-1892181*x^3+376428*x^2+2193504*x+51904
173857  1937  38191384824694383099923729    
1      1      8       0.114892   0.183156      0.188276
P1=x^3-27*x^2+35*x-1,P2=x^3-11*x^2+8*x-1,P=x^9+48*x^8+303*x^7-10953*x^6
   -72549*x^5+825678*x^4+1083824*x^3-357201*x^2-414609*x+57421
10309   1929  7864050646576255644981    
2      1      11      0.300038   0.183256      0.302464
P1=x^3-18*x^2+31*x-1,P2=x^3-30*x^2+43*x-1,P=x^9-36*x^8-426*x^7+18708*x^6
   +66213*x^5-2207940*x^4-1980725*x^3+5522748*x^2+2482560*x+22464
178889   1261265  11486029882117782845780928107151625    
2      3      17      0.229243   0.296055      0.300498
     p=3
P1=x^3-47*x^2+27*x-1,P2=x^3-14*x^2+26*x-1,P=x^9+99*x^8+2110*x^7-39581*x^6
   -841754*x^5+12433359*x^4-31915251*x^3+12891832*x^2+16161948*x+8084
284788   57741  4446496553844548173991089269312    
1      1      7       0.174945   0.200408      0.217948
P1=x^3-31*x^2+25*x-1,P2=x^3-24*x^2+38*x-1,P=x^9+21*x^8-1152*x^7-17265*x^6
   +370464*x^5+2658657*x^4-5851191*x^3-1210464*x^2+3554288*x+55138
432884   573349  15288742990049019447046087884332096    
1      1      10      0.169300   0.165712      0.279145
P1=x^3-22*x^2+41*x-1,P2=x^3-9*x^2+18*x-1,P=x^9+39*x^8+288*x^7-3470*x^6
   -23571*x^5+176589*x^4-88881*x^3-684987*x^2+578139*x-18043
511537  321  4427374441992552457143633    
2      1      14      0.334301   0.380706      0.542048
P1=x^3-23*x^2+35*x-1,P2=x^3-24*x^2+30*x-1,P=x^9-3*x^8-906*x^7+1667*x^6
   +206130*x^5-144453*x^4-552539*x^3+378690*x^2+168384*x-876
110580   368037  7489652934283408190167772904000    
1      1      7       0.189195   0.171444      0.216350
P1=x^3-23*x^2+17*x-1,P2=x^3-36*x^2+27*x-1,P=x^9-39*x^8-1017*x^7+37436*x^6
   +322812*x^5-7556721*x^4-95099*x^3+3294255*x^2-9906*x-2367
91572   77433  39611733265845206525895660864    
1      1      8       0.192319   0.195184      0.266941
\end{verbatim}
\normalsize

\begin{theorem} Let $K$ be a totally real number field and let
${\mathcal K}^{\rm c}$ be the set of subfields $K_n$ of the $p$-cyclotomic 
tower $K^{\rm c}$ of $K$ (with $[K_n : K]=p^n$, for all $n \geq 0$). Then, 
under the Leopoldt conjecture in $K^{\rm c}$, $C_p(K_n) \too 0$ as $n \to\infty$.
\end{theorem}

\begin{proof} From \cite[\S\,3, Proposition 2]{Wa2}, we get $\sqrt{D_{K_n}} \geq
p_{}^{\alpha\cdot n \cdot p^n+O(p^n)}$ with $\alpha >0$; then from Iwasawa's theory, there
exist $\lambda, \mu \in \N$ and $\nu \in \Z$ such that $\order {\mathcal T}_{K_n}
=p_{}^{\lambda\,n + \mu\,p^n + \nu}$ for $n\gg 0$. 
So we obtain
$C_p(K_n) \leq \Frac{\lambda\,n + \mu\,p^n + \nu}{\alpha\cdot n \cdot p^n+O(p^n)}$
for $n \gg 0$, where the limit of the upper bound is $0$; whence the result giving 
an example of family (${\mathcal K}^{\rm c}$) for which the 
Conjecture \ref{conjprinc} is verified.
\end{proof}

Note that if $K \in {\mathcal K}_{\rm real}$ is $p$-rational (i.e., $C_p(K)=0$), 
then $C_p(K_n)=0$ for all $n \geq 0$: see \cite{Gr1}, Proposition IV.3.4.6
from the formula of invariants (Theorem 3.3) giving $C_p(L)=0$
for any $p$-primitively ramified $p$-extension $L$ of $K$ (Definition 3.4).

\begin{remark}
In \cite{HM}, Hajir and Maire define, in the spirit of an algebraic $p$-adic
Brauer--Siegel theorem, the {\it logarithmic mean exponent} of a finite 
$p$-group $A \simeq \prd_{i=1}^r \Z/p^{a_i}\Z$, by the formula
$\M_p(A) := \Frac{1}{r} \cdot \Frac{{\rm log}_\infty (\order A)}{{\rm log}_\infty(p)}
= \Frac{1}{r} \sm_{i=1}^r a_i = \Frac{1}{r} \cdot v_p(\order A)$,
and applied to tame generalized class groups.
In the case of ${\mathcal T}_K$, we get $v_p(\order {\mathcal T}_K) =
{\rm rk}_p({\mathcal T}_K) \cdot \M_p({\mathcal T}_K)$, and 
we would have conjecturally, for any $K \in {\mathcal K}_{\rm real}$:
\ctlm{$\M_p({\mathcal T}_K) \leq {\mathcal C}_p \cdot \Frac{1}{{\rm rk}_p({\mathcal T}_K)} \cdot
\Frac{{\rm log}_\infty(\sqrt{D_K})}{{\rm log}_\infty(p)} \leq {\mathcal C}_p \cdot 
\Frac{{\rm log}_\infty(\sqrt{D_K})}{{\rm log}_\infty(p)}$.}
But in \cite[Theorems 0.1, 1.1, Proposition 2.2]{HM}, this function $ \M_p$ is 
essentially used for class groups in particular infinite towers with tame 
restricted ramification for which some explicit upper bounds are obtained.
\end{remark}

In this context, we can suggest the following direction of search:

\begin{proposition} Let $K$ be a totally real number field and let $L$ be
the (totally  real) $p$-Hilbert tower of $K$; we assume that $L/K$ is infinite. 
Let ${\mathcal K}$ be a set of subfields $K_n$ of $L$, with $K_n \subset K_{n+1}$
and $[K_n : K]=p^n$ for all $n \geq 0$. 

\noindent
Then $C_p(K_n) = \Frac{v_p(\order {\mathcal T}_{K_n}) \cdot {\rm log}_\infty(p)}
{p^n \cdot {\rm log}_\infty(\sqrt{D_{K}})}$, and Conjecture \ref{conjprinc} 
is true for ${\mathcal K}$
as soon as $v_p(\order {\mathcal T}_{K_n})$ is ``essentially'' a linear function of 
the degree $[K_n : K]=p^n$ as $n \to\infty$ (i.e., $v_p(\order {\mathcal T}_{K_n}) =
\alpha \, n + \beta\,p^n + \gamma$ for all $n \gg 0$,
$\alpha, \beta \in \N$, $\gamma \in \Z$).
\end{proposition}

\begin{proof} 
Since $K_n/K$ is unramified, 
$D_{K_n} = D_{K}^{[K_n : K]}\! \cdot {\rm N}_{K/\Q} (D_{K_n/K})=D_K^{p^n}$. 
So, for all $n \gg 0$, $C_p(K_n) = \Frac{(\alpha \, n + \beta\,p^n + \gamma) \cdot {\rm log}_\infty(p)}
{p^n \cdot {\rm log}_\infty(\sqrt{D_{K}})}$, equivalent to the constant
$\Frac{\beta\cdot {\rm log}_\infty(p)}{{\rm log}_\infty(\sqrt{D_{K}})}$ at infinity.
Whence the existence of ${\mathcal C}_p$ over ${\mathcal K}$.
If $\beta=0$, then $C_p(K_n) \to 0$.
\end{proof}

\noindent
The orders $\order \Cl_{K_n}$ have this property of ``linearity'' 
and ${\rm rk}_p(\Cl_{K_n}) \to \infty$ under some conditions \cite[Theorem A]{Ha}; 
thus, it would remain the question of a similar linearity for the valuations, 
according to $[K_n : K]$, of the normalized regulators ${\mathcal R}_{K_n}$.

\subsection{Comparison ``archimedean'' versus ``$p$-adic''}\label{avsp}

The above considerations are, in some sense, a $p$-adic approach 
of some deep results (Brauer--Siegel--Tsfasman--Vlad\u{u}\c{t} 
theorems \cite{TV, Zy} and broad generalizations in \cite{Ts}, then
\cite{L1} for quantitative bounds from the Brauer--Siegel theorem) 
on the behavior, 
in a tower $L := \bigcup_{n \geq 0} K_n$ of finite extensions $K_n/K$, 
of the quotient $B\!S_{K_n}:= \Frac{{\rm log}_\infty(h_{K_n}\! \cdot 
R_{K_n, \infty})}{{\rm log}_\infty(\sqrt{D_{K_n}})}$.

\noindent
Of course, in order to infer the $p$-adic case, our purpose is 
to deal, in the archimedean one, with any $K \in {\mathcal K}_{\rm real}$ 
or with families ${\mathcal K}$ fulfilling some specific conditions (e.g., $[K : \Q]=d$, 
$\Frac{[K : \Q]}{{\rm log}_\infty(\sqrt{D_K})} \to 0$), which is 
possible thanks to \cite[Theorem 1]{Zy}, at least for Galois fields. 
For any $K \in {\mathcal K}_{\rm real}$, let
$B\!S_K := \Frac{{\rm log}_\infty(h_K \cdot R_{K,\infty})}{{\rm log}_\infty(\sqrt{D_K})}$.

We shall consider the following {\it normalized quotient} $\wt {B\!S}_K = B\!S_K-1$
using $\order {\mathcal T}_{K,{p_\infty}}$ instead of $h_K \cdot R_{K,\infty}$:
\begin{equation}\label{BSTV}
\wt {B\!S}_K :=   \Frac{{\rm log}_\infty \Big(h_K \cdot \Frac{R_{K,\infty}}{\sqrt{D_K}} \Big)}
{{\rm log}_\infty(\sqrt{D_K})} =  \Frac{{\rm log}_\infty(\order {\mathcal T}_{K,{p_\infty}})}
{{\rm log}_\infty(\sqrt{D_K})}, \  K \in {\mathcal K} \ \hbox{(from formula \eqref{infini})},
\end{equation}
and presume that this function is bounded over ${\mathcal K}$. 
When the degree is constant in the family, the classical Brauer--Siegel theorem
applies since $\Frac{[K : \Q]}{{\rm log}_\infty(\sqrt{D_K})} \to 0$.

The following program gives, for the family ${\mathcal K}_{\rm real}^{(2)}$ 
of real quadratic fields of discriminants $D$, 
consistent verifications for the original function $B\!S$:

\smallskip
\footnotesize
\begin{verbatim}
{Max=0;Min=1;for(D=10^8,10^8+10^6,e=valuation(D,2);M=D/2^e;if(core(M)!=M,next);
if((e==1||e>3)||(e==0 & Mod(M,4)!=1)||(e==2 & Mod(M,4)==1),next);P=x^2-D;
K=bnfinit(P,1);C8=component(K,8);h=component(component(C8,1),1);
reg=component(C8,2);BS=log(h*reg)/log(sqrt(D));if(BS<Min,Min=BS;
print(D," ",Min," ",Max));if(BS>Max,Max=BS;print(D," ",Min," ",Max)))}
\end{verbatim}
\normalsize

\smallskip\noindent
$0.647 < B\!S < 1.155$ for $D \in [10^5, 2 \!\cdot\! 10^5]$, $0.734 < B\!S < 1.136$ 
for $D \in [10^7, 10^7+10^5]$, $0.7657< B\!S < 1.1239$ for $D \in [10^8, 10^8+10^5]$,
and $0.75738 < B\!S < 1.12713$ for $D \in [10^8, 10^8+10^6]$ (more than two days of computer),
showing:  $\wt {B\!S}_K=O(1) < 1$. 
Then $0.773 < B\!S < 1.113$ for the family $K=\Q(\sqrt {a^2+1})$, $a \in [10^4, 2 \cdot 10^4]$.

\smallskip \noindent
In the same way, the family ${\mathcal K}_{\rm ab}^{(3)}$ of
cyclic cubic fields of conductors $f$, gives:

\smallskip\noindent
$0.6653 \leq B\!S \leq 1.1478$ for $f \in [10^4, 10^6]$, 
$0.7547 \leq B\!S \leq 1.1385$ for $f \in [10^6, 2 \!\cdot \!10^6]$.

\begin{remarks} (i)
In the archimedean viewpoint, we have
 $C_{p_\infty}(K)=\Frac{{\rm log}_\infty(\order {\mathcal T}_{K,{p_\infty}})}
{{\rm log}_\infty(\sqrt{D_K})}$, giving, from the expression \eqref{BSTV} 
of $\wt {B\!S}_K$, ${\rm log}_\infty(\order {\mathcal T}_{K,{p_\infty}}) 
= \wt {B\!S}_K \cdot {\rm log}_\infty(\sqrt{D_K})$;
thus we obtain about the above calculations for 
the examples of fixed families ${\mathcal K}$:
\ctl{${\rm log}_\infty(\order {\mathcal T}_{K,{p_\infty}}) \leq 
O(1) \cdot {\rm log}_\infty(\sqrt{D_K})$
 written  ${\rm log}_\infty(\order {\mathcal T}_{K,{p_\infty}}) \leq 
{\mathcal C}_{p_\infty} \cdot {\rm log}_\infty(\sqrt{D_K})$,}
giving, in some sense, the inequality of the $p$-adic Conjecture 
\ref{conjprinc} with the audacious convention for the 
{\it infinite place} ${p_\infty}$ and 
${\mathcal T}_{K,{p_\infty}}=h_K \cdot \Frac{R_{K,\infty}}{\sqrt{D_K}}$: 

\centerline{${\rm log}_\infty({p_\infty})=1 \ \ \& \ \ 
v_{{p_\infty}} (\order {\mathcal T}_{K,{p_\infty}}) = 
{\rm log}_\infty(\order {\mathcal T}_{K,{p_\infty}})$.}

\smallskip \noindent
in which case, the constant ${\mathcal C}_{p_\infty}$ is the maximal value
reached by $\wt {B\!S}_K = B\!S_K-1$ over the given family ${\mathcal K}$.

\medskip
(ii) One may wonder about the differences of behaviour and properties 
between $C_{p_\infty}(K)$ and $C_{p}(K)$, as $D_K \to \infty$, 
because of the choosen normalizations 
and the role of the discriminant in the definitions. The only change could be to define:
\ctlm{${\mathcal T}'_{K,{p_\infty}}=h_K \cdot R_{K,\infty}\ $ and 
$\ C'_{p_\infty}(K) = \Frac{{\rm log}_\infty({\mathcal T}'_{K,{p_\infty}})}
{{\rm log}_\infty(\sqrt {D_K})} = C_{p_\infty}(K)+1 = {B\!S}_K$,}
by reference to Brauer--Siegel context,
but in that case, we should have (from \eqref{residu})
${\mathcal T}'_{K,{p_\infty}} =\wt \kappa_{K,{p_\infty}} \cdot \sqrt {D_K}$,
with $\wt \kappa_{K,{p_\infty}} = \frac{1}{2^{d-1}} \cdot \kappa_{K,{p_\infty}}$,
which cannot be a suitable normalization of the $\zeta$-function and its residue;
indeed, on the interval $[2, 10^6]$ of discriminants of real quadratic fields,
the local maxima of 
$(\kappa_{K,{p_\infty}}, \kappa_{K,{p_\infty}} \cdot \sqrt {D_K})$
increase excessively from $(0.215204,\  0.481211)$ to $(2.732814,\  2705.305810)$.

\smallskip
But the comparison must take into account the 
difference of nature of the sets of values of the functions $C_{p_\infty}$ and $C_p$:

\smallskip
The first one takes its values in an explicitely bounded interval of $\R$, containing~$0$,
given by the Brauer--Siegel--Tsfasman--Vlad\u{u}\c{t}--Zykin results:
\ctl{$S_{p_\infty}= \Big \{ v_{p_\infty}(\order {\mathcal T}_{K,\infty}) \cdot
\Frac{{\rm log}_\infty({p_\infty})}{{\rm log}_\infty(\sqrt{D_K})},\  K \in {\mathcal K} \Big\} 
\subseteq \R \cdot \Frac{{\rm log}_\infty({p_\infty})}{{\rm log}_\infty(\sqrt{D_K})}$,}
while the second one takes its values in a discrete set of the form:
\ctl{$S_p= \Big \{ v_p(\order {\mathcal T}_{K,p}) \cdot 
\Frac{{\rm log}_\infty(p)}{{\rm log}_\infty(\sqrt {D_K})},\  K \in {\mathcal K} \Big\} 
\subseteq \N \cdot \Frac{{\rm log}_\infty(p)}{{\rm log}_\infty(\sqrt {D_K})}$,}
so that $v_{p_\infty}(\order {\mathcal T}_{K,\infty}) = 
{\rm log}_\infty(\order {\mathcal T}_{K,\infty})$ is never $0$ (except if $K=\Q$)
while $v_p(\order {\mathcal T}_{K,p})$ is equal to $0$ for infinitely many fields $K$, probably
with a positive density which increases significantly as $p\to \infty$; but, symmetrically,
we have seen that the integers $v_p(\order {\mathcal T}_{K,p})$ take infinitely 
many strictly positive values for huge discriminants. 

\smallskip
To compare the two situations one must probably compute some ``integrals'' when 
$D_K$ varies in some intervals. Whatever the choice of the family ${\mathcal K}$, 
the sets of real coefficients $\Frac{{\rm log}_\infty(p_v)}{{\rm log}_\infty(\sqrt {D_K})}$ are 
homothetic discrete subsets of $\R_+$ as $v$ varies, so that the 
comparison is based on the coefficients $v_{p_\infty}(\order {\mathcal T}_{K,\infty})$ \&
$v_p(\order {\mathcal T}_{K,p})$, respectively.

\smallskip
The following programs compute the means of $C_{v}(K)$ on intervals 
of discriminants $D_K$, $K \in{\mathcal K}_{\rm real}^{(2)}$, for $p_\infty$ and $p \geq 2$, 
but many other means may be interesting:

\smallskip\footnotesize
\begin{verbatim}
{Sinfty=0.0;N=0;for(D=10^5,2*10^5,e=valuation(D,2);M=D/2^e;if(core(M)!=M,next);
if((e==1||e>3)||(e==0 & Mod(M,4)!=1)||(e==2 & Mod(M,4)==1),next);P=x^2-D;
N=N+1;K=bnfinit(P,1);C8=component(K,8);h=component(component(C8,1),1);
reg=component(C8,2);Cp=log(h*reg)/log(sqrt(D))-1;Sinfty=Sinfty+Cp);print(Sinfty/N)}

{p=3;n=18;Sp=0.0;N=0;for(D=10^5,2*10^5,e=valuation(D,2);M=D/2^e;if(core(M)!=M,next);
if((e==1 || e>3)||(e==0 & Mod(M,4)!=1)||(e==2 & Mod(M,4)==1),next);P=x^2-D;
N=N+1;K=bnfinit(P,1);Kpn=bnrinit(K,p^n);C5=component(Kpn,5);
Hpn0=component(C5,1);Hpn=component(C5,2);Hpn1=component(Hpn,1);
vptor=valuation(Hpn0/Hpn1,p);Cp=vptor*log(p)/log(sqrt(D));Sp=Sp+Cp);print(Sp/N)}
\end{verbatim}
\normalsize

\footnotesize
$v=p_\infty$ gives 
$M_\infty=-0.08072025$ for $D \in [5,10^6]$, $M_\infty=-0.05566364$ 
for $D \in [10^8,10^8+10^5]$

\noindent
\hspace{2.15cm} $M_\infty=-0.06817971$ for $D \in [5,10^7]$,
$M_\infty=-0.05562784$ for $D \in [10^8,10^8+10^6]$

\hspace{6.55cm} $M_\infty=-0.04947600$ for $D \in [10^9,10^9+10^4]$

\noindent
$p=3$ ($n=18$) gives $\ M_3=0.12656432$ for $D \in [5,10^6]$, $M_3=0.10463765 $
for $D \in [10^7,10^7+10^5]$

\noindent
$p=5$ ($n=12$) gives $\ M_5= 0.07257764 $ for $D \in [5,10^6]$, $M_5= 0.05897703 $
for $D \in [10^7,10^7+10^5]$

\noindent
$p=7$ ($n=10$) gives $\ M_7= 0.05647554$ for $D \in [5,10^6]$, $M_7= 0.04649732$ 
for $D \in [10^7,10^7+10^5]$

\noindent
$p=29$ ($n=6$) gives $M_{29}= 0.01901355$ for $D \in [5,10^6]$, $M_{29}= 0.01572121$ 
for $D \in [10^7,10^7+10^5]$
\end{remarks}

\normalsize \noindent
giving obvious heuristics about the behaviour of each mean.

\section{Conclusions}
The analysis of the archimedean case, depending on the properties of 
the complex $\zeta$-function of $K$, is sufficiently significant to hope the 
relevance of the $p$-adic one for which we give some observations, 
despite the lack of proofs:

\medskip \noindent
\quad ({\rm\bf a})
In the $p$-adic Conjecture \ref{conjprinc}, 
the most important term is $v_p(\order {\mathcal R}_{K,p})$, the valuation  
of the normalized $p$-adic regulator, the contribution of $v_p (\order \Cl_{K,p})$ 
being probably negligible compared to $v_p (\order {\mathcal R}_{K,p})$
as shown, among other, by classical heuristics \cite{CL,CM}, and
reinforced by the recent conjectures cited in the \S\,\ref{vph}\,(ii).

\smallskip
Furthermore, for $K$ fixed, $v_p (\order \Cl_{K,p}) \geq 1$ for finitely many 
primes $p$, but the case of $v_p(\order {\mathcal R}_{K,p})$
is an out of reach conjecture \cite[Conjecture 8.11]{Gr2}.

\medskip \noindent
\quad ({\rm\bf b})
The family of Subsection \ref{family} shows that $p$-adic regulators 
may tend $p$-adically to $0$, even in simplest cases, and it should be  
of great interest to find other such critical sub-families of units,
depending on arbitrary large $p$-powers, to precise the relation 
between $v_p(\order {\mathcal R}_{K,p})$ and 
${\rm log}_\infty(\sqrt{D_K})$, $K \in {\mathcal K}_{\rm real}^{(d)}$, 
for degrees $d>2$.

\smallskip
After the writing of this paper we have found the reference \cite{Wa5} 
about the family of cyclic cubic fields $K$ defined by 
$P=x^3-(N^3-2 N^2+3 N-3) \, x^2-N^2 x-1$ for any $N \in \Z$, $N \ne 1$, 
near $1$ in $\Z_3$; this paper of Washington deals with $p=3$, 
to obtain $3$-adic $L$-functions with zeros arbitrarily close to $1$, but
we observed that any $p \geq 2$ gives interesting non-$p$-rational 
fields with large $v_p(\order {\mathcal T}_{K,p})$ and $C_p(K)<1$ for all. 
The reader may play with the following program (choose $p \geq 2$, 
the intervals defining $N=1+a\,p^k$, a lower bound
$vp$ for $vptor$ and $n$ large enough):

\footnotesize
\begin{verbatim}
{p=2;bk=2;Bk=10;ba=1;Ba=12;vp=10;n=36;print("p=",p);for(k=bk,Bk,for(a=ba,Ba,
if(Mod(a,p)==0,next);N=1+a*p^k;P=x^3-(N^3-2*N^2+3*N-3)*x^2-N^2*x-1;K=bnfinit(P,1);
Kpn=bnrinit(K,p^n);C5=component(Kpn,5);Hpn0=component(C5,1);Hpn=component(C5,2);
Hpn1=component(Hpn,1);vptor=valuation(Hpn0/Hpn1,p);
if(vptor>vp,D=component(component(K,7),3);Cp=vptor*log(p)/log(sqrt(D));
print("a=",a," k=",k," D=",D," vptor=",vptor," Cp=",Cp);print("P=",P))))}
\end{verbatim}
\normalsize

\noindent
giving for instance the interesting cases with $a=1$ ($p=2$, $3$, $5$):

\footnotesize
\begin{verbatim}
p=2 k=9 D=17213619969^2  vptor=28  Cp=0.8234   P=x^3-134480895*x^2-263169*x-1
p=3 k=9 D=150102262056706213^2  vptor=23  Cp=0.6388
                                               P=x^3-7625984944841*x^2-387459856*x-1
p=5 k=5 D=95397978509379^2 vptor=10 Cp=0.4999  P=x^3-30527349999*x^2-9771876*x-1
\end{verbatim}
\normalsize

\smallskip\noindent
\quad ({\rm\bf c}) 
Consider, for any $p \geq 2$ and any $K \in {\mathcal K}_{\rm real}$:
\ctlm{$C_p(K):= 
\Frac{v_p(\order {\mathcal T}_{K,p}) \cdot {\rm log}_\infty(p)}{{\rm log}_\infty(\sqrt{D_K})}$,
\ \ ${\mathcal C}_p := \ds\sup_K(C_p(K))$, \ \ ${\mathcal C}_K :=\ds \sup_p(C_p(K))$.}
\qquad (i) The existence of ${\mathcal C}_K < \infty$, for a given $K$, only says that
the conjecture proposed in \cite[Conjecture 8.11]{Gr2}, claiming that any number 
field is $p$-rational for all $p \gg 0$, is true for the field $K$; for this field,
$\ds\limsup_{p}(C_p(K)) = 0$.

\noindent
\qquad (ii) If ${\mathcal C}_p$ does exist for a given $p$,
we have an universal $p$-adic analog of Brauer--Siegel theorem 
(Conjecture \ref{conjprinc}).
The existence of ${\mathcal C}_p < \infty$ may be true
taking instead $\ds\sup_{K \in {\mathcal K}} (C_p(K))$, for particular families
${\mathcal K}$ (e.g., extensions of fixed degree or subfields of
some infinite towers as in \cite{HM, Iv, TV, Zy});
but we must mention that for the invariants ${\mathcal T}_{K,p}$, the
transfer map ${\mathcal T}_{K,p} \too {\mathcal T}_{L,p}$ is injective
in any extension $L/K$ in which Leopoldt's conjecture is assumed
\cite[Theorem IV.2.1]{Gr1}, which leads to a major difference 
from the case of $p$-class groups. 

\smallskip\noindent
\qquad (iii) Furthermore, it seems that $\ds\limsup_{K\in {\mathcal K}}(C_p(K))$ 
may be $\leq 1$ for any $p$; then $\ds\limsup_{p}(C_p(K)) = \infty$ or $0$, 
for any $K$, depends on \cite[Conjecture 8.11]{Gr2}. But computations 
for very large discriminants (of a great lot of quadratic fields for 
instance) is out of reach (see the Remarks of the \S\,\ref{vph}).

\medskip \noindent
\quad ({\rm\bf d}) 
When $p$ and $D_K$ are not independent, this yields some 
interesting potential results as the following one:
let ${\mathcal K}_{\rm real}(p^e)$ be the set of fields $K \in {\mathcal K}_{\rm real}$ 
of discriminant $D_K=p^e$, for any fixed $p$-power $p^e$, $e \geq 1$; 
then, as soon as $C_p(K)< \frac{2}{e}$ for all $K$ in some subfamily 
${\mathcal K}(p^e)$ of ${\mathcal K}_{\rm real}(p^e)$, $K$ is $p$-rational
since then $C_p(K) =  \frac{2}{e} \cdot v_p(\order {\mathcal T}_{K})$.
For instance, if we were able to prove that $C_p(K)<2$ for all 
$K \in {\mathcal K}_{\rm real}^{(2)}(p)$ (quadratic fields $K=\Q(\sqrt p)$, 
$p \equiv 1 \pmod 4$), this would imply the conjecture of 
Ankeny--Artin--Chowla (see \cite[\S\,5.6]{Wa1}), affirming that 
$\varepsilon_K =: u + v\,\sqrt p$ is such that $v \not\equiv 0
\pmod p$, which is equivalent, since $\Cl_K=1$, to ${\mathcal R}_K \sim 1$ 
(indeed, $\varepsilon_K^p \equiv u \equiv \varepsilon_K - v\,\sqrt p \pmod p$,
whence $\varepsilon_K^{p-1} \equiv 1+ \varepsilon_K^{-1}v\,\sqrt p  \pmod p$). 

The cyclic quartic fields of conductor $p$ (i.e., 
$K \in {\mathcal K}_{\rm real}^{(4)}(p^3)$) give no solution in the 
selected interval, although $C_p(K)=\frac{2}{3}\,v_p(\order {\mathcal T}_{K})$.
The case of $K \in {\mathcal K}_{\rm real}^{(3)}(p^2)$ (cyclic cubic fields of 
conductor $p$) is interesting since, in this case, 
$C_p(K)=v_p(\order {\mathcal T}_{K})$, for which
$v_p(\order {\mathcal T}_{K})=1$ is more credible if we consider 
that for instance $C_p(K)<2$ over ${\mathcal K}_{\rm real}^{(3)}(p^2)$; 
indeed we have found only two examples up to $p \leq 10^8$:  

\smallskip
\footnotesize
\begin{verbatim}
p=5479       vptor=Cp=1    P=x^3 + x^2 - 1826x + 13799           
p=15646243   vptor=Cp=1    P=x^3 + x^2 - 5215414x - 311765879  
\end{verbatim}
\normalsize

Let's give few examples in degrees $d=5, 7, 9$ using $polsubcyclo(p,d)$ (cyclic 
fields of conductor $p$) since $C_p(K)=\frac{2}{d-1}\, v_p(\order {\mathcal T}_{K})$
(for all, $v_p(\order {\mathcal T}_{K})=1$, $v_p(\order \Cl_K)=0$):

\footnotesize
\begin{verbatim}
p=130811  Cp=0.5000  P=x^5+x^4-52324*x^3-429060*x^2+575263872*x+3600157696
p=421     Cp=0.3333  P=x^7+x^6-180*x^5-103*x^4+6180*x^3+11596*x^2-25209*x-49213
p=44563   Cp=0.3333  P=x^7+x^6-19098*x^5-87307*x^4+73981206*x^3-1061790574*x^2
                                                         -13438850605*x-28465212577
p=37      Cp=0.2500  P=x^9+x^8-16*x^7-11*x^6+66*x^5+32*x^4-73*x^3-7*x^2+7*x+1
p=13411   Cp=0.2500  P=x^9+x^8-5960*x^7+117167*x^6+5761671*x^5-114461957*x^4
                       -2103829198*x^3+33776243778*x^2+244391306047*x-3339737282887
\end{verbatim}
\normalsize

In other words, a more general ``Ankeny--Artin--Chowla Conjecture'' should be
that the set of non-$p$-rational
$K \in {\mathcal K}_{\rm real}^{(d)}(p^e)$ (or any suitable subfamily) is finite.
Thus the existence (if so), and then the order of magnitude of ${\mathcal C}_p$, 
would govern many obstructions and/or finiteness theorems in number theory.

\medskip \noindent
\quad ({\rm\bf e}) 
On another hand, the difficult Greenberg's conjecture \cite{Gre1}, 
on the triviality of the Iwasawa invariants $\lambda$, $\mu$ for the $p$-class 
groups in $K^{\rm c}$, in the totally real case, goes in the sense of rarity 
of large $p$-class groups as we have mentioned at the \S\,\ref{vph}\,(ii),
and this conjecture also depends on Fermat's quotients of algebraic 
numbers (\cite[\S\,7.7]{Gr3}, \cite[\S\,4.2]{Gr6}) or of a similar 
logarithmic framework as in \cite{Ja}.
In the same way, some other conjectures of Greenberg \cite{Gre2} depend,
in a crucial manner, of the existence of $p$-rational fields with given Galois groups.  

\medskip \noindent
\quad ({\rm\bf f}) 
But all this is far to be proved because of a terrible lack of knowledge 
of $p$-Fermat quotients of algebraic numbers, a notion which gives a
{\it weaker information} than the $p$-adic logarithms or regulators, but which
governs many deep arithmetical problems, even assuming the Leopoldt
conjecture which appears as a rough step in the study of 
${\rm Gal}(H_K^{\rm pr}/K)$; indeed, if Leopoldt's conjecture is not fulfilled
in a given field $K$,
there exists a sequence $\varepsilon_i \in E_K$, $\varepsilon_i \notin E_K^p$,
such that $\delta_p(\varepsilon_i)\to \infty$ with $i$, which shows the 
extreme uncertainty about the ${\mathcal T}_{K,p}$ groups.

\smallskip \noindent
\quad ({\rm\bf g}) 
Recal to finish that ${\mathcal T}_{K,p}$ is the dual of 
${\rm H}^2(G_p(K), \Z_p)$ (\cite[Chapitre 1]{Ng}, then \cite[Appendix, Theorem 2.2]{Gr1}), 
where $G_p(K)$ is the Galois group of the maximal $p$-ramified pro-$p$-extension of $K$
(for which $G_p(K)^{\rm ab} \simeq \Z_p \times {\mathcal T}_{K,p}$ in the 
totally real case, under Leopoldt's conjecture), and can be considered as 
the first of the still mysterious non positive twists ${\rm H}^2 (G_p(K), \Z_p (i))$ 
of the motivic cohomology (whereas the positive twists can be dealt with 
using K-theory thanks to the Quillen--Lichtenbaum conjecture, now a 
theorem of V\oe vodsky--Rost and al.).

\subsection*{Acknowledgments} I thank Christian Maire for discussions 
about some aspects of Brauer--Siegel--Tsfasman--Vlad\u{u}\c{t} theorems, 
St\'ephane Louboutin for references on complex $\zeta$-functions, 
Thong Nguyen Quang Do for confirming to me the critical role of 
${\mathcal T}_{K,p} \simeq {\rm H}^2 (G_p(K), \Z_p (0))^*$, from the 
cohomological viewpoint recalled above.

\end{document}